\documentclass[reqno,11pt,a4paper]{amsart}

\usepackage[margin=2.25cm]{geometry}

\newtheorem{theorem}{Theorem}[section]
\newtheorem{corollary}[theorem]{Corollary}
\newtheorem{lemma}[theorem]{Lemma}
 
\theoremstyle{definition}

\theoremstyle{remark}

\numberwithin{equation}{section}

\usepackage{amsmath,amssymb,amsthm}
\usepackage{amsfonts}
\usepackage{mathrsfs}

\def\mc{\omega}
\def\cE{{\mathcal E}}
\def\cM{{\mathcal M}}

\def\bR{{\mathbb R}}

\def\rD{\mathrm{D}}
\def\rd{\mathrm{d}}

\def\cC{\mathcal{C}}
\def\cF{\mathcal{F}}

\def\cS{{\mathcal S}}
\newcommand\cScmd[1]{\cS(#1)}
\def\cB{{\mathcal B}}

\newcommand{\solM}[1]{\mathsf{Sol}\!\left(#1\right)}

\def\uu{w_v}
\def\uup{w_v'}

\def\pa{\partial}
\def\rhob{\rho}
\def\ve{{\mathbf e}}
\def\qq{i}

\def\Om{\Omega}
\def\om{\omega}
\def\Oms{\cB_r}
\def\Omrho{\Om_\rho}
\def\ep{\varepsilon}
\def\fm{f_{\max}}
\def\detn{\mathrm{det}^{\frac1n}}
\def\detns{\mathrm{det}^{\frac2n}}
\def\detni{\mathrm{det}^{-\frac1n}}
\def\diam{\mathrm{diam}}
\def\dist{\mathrm{dist}}
\def\inin{\textnormal{\textsf{in}}}

\def\cH{\mathcal{H}}
\def\qin{q_\inin}
\def\Qin{Q_\inin}
\def\Cmc{C_{\mc}}

\def\Tin{\tau}
\def\Tini{\Tin^{-1}}
\def\bTin{\br{\tau}}
\def\bTini{\br{\tau}^{-1}}
\def\xs{{\mathsf x}}
\def\xm{\xs_{\mathsf m}}
\def\ds{\bar d}
\def\Vo{V}
\def\Fo{\Phi}

\def\Cht{{\widehat C}}
\def\wt{\widetilde}

\def\pp{{\mathbf P}}
\def\hc{h_{\textnormal c}}

\newcommand{\br}[1]{\breve{#1}}
\newcommand{\Ompa}[1]{\Om_{\langle #1 \rangle}}
\newcommand\qu[1]{\quad\text{#1}\;\;}
\newcommand\qua[1]{\quad\text{#1}\quad}
\newcommand\qqu[1]{\qquad\text{#1}\quad}
\newcommand\qqua[1]{\qquad\text{#1}\qquad}
\newcommand\mat[1]{\|#1\|}
\newcommand\mt[1]{[#1]}
 
\def\be{\begin{equation}}
\def\ee{\end{equation}}

\usepackage{color,hyperref}
\definecolor{darkblue}{rgb}{0.0,0.0,0.85}
\definecolor{darkgreen}{rgb}{0,0.5,0}
\hypersetup{colorlinks,breaklinks,
            linkcolor=darkblue,urlcolor=darkblue,
            anchorcolor=darkblue,citecolor=darkblue}

\title{Interior Estimates for Monge-Amp\`ere Equation in Terms of Modulus of Continuity}
\author{Bin Cheng}
\address{Department of Mathematics, University of
          Surrey, Guildford, GU2 7XH, United Kingdom}
   \email{b.cheng@surrey.ac.uk}
   
\author{Thomas O'Neill}
\address{Department of Mathematics, University of
          Surrey, Guildford, GU2 7XH, United Kingdom}
   \email{t.o'neill@surrey.ac.uk}
\date{5 May, 2020}
\begin{document}
\maketitle
\begin{abstract}
We  investigate the Monge-Amp\`ere equation subject to zero boundary value and with a positive right-hand side  function assumed to be continuous or essentially bounded. Interior estimates of the solution's first and second derivatives are obtained in terms of moduli of continuity. We explicate how the estimates depend on various quantities but have them independent of the solution's modulus of {\it convexity}. Our main theorem has many useful consequences. One of them is the {\it nonlinear} dependence between the H\"older seminorms of the solution and of the right-side function, which confirms the results of Figalli, Jhaveri \& Mooney in \cite{AFCM1}. Our technique is in part inspired by Jian \& Wang in \cite{XJW2} which includes using a sequence of so-called sections.
\end{abstract}
\section{\bf Introduction}
In an open and bounded convex set $\Om$ in $\bR^n$ with $n\ge 2$, consider a strictly convex function $v \in \cC^2(\Om) \cap \cC^0(\overline{\Om})$ that satisfies the following Monge-Amp\`ere equation  (MAE)
\begin{align}
     \detn \mathrm{D}^2v &= f \qu{    in    }\Om \label{MA1}, \\
      v &=0 \qu{    on    }\pa \Om \label{MA2} 
 \end{align}
 with $f$ bounded from above and below by finite positive constants. Here $\detn A:=(\det A)^{\frac1n}$.

In the study of second order elliptic PDEs, ranging from the   Laplacian equation $\Delta v=f$ to the fully nonlinear cases such as the MAE, a fundamental question is: can the solution $v$ recover  two derivatives of regularity from the right-hand side function $f$, at least over the interior of the domain? Caffarelli in  \cite{LC5} proved affirmative answer to this question for \eqref{MA1}, \eqref{MA2} in the cases of Lebesgue integrable $f$  and H\"older continuous $f$. These are regularity results without estimates. More than two decades later, Figalli et al. proved in \cite{AFCM1}  interior estimates for MAE in terms of  H\"older norms of  $\rD^2 v$ and $f$. In particular they proved that both upper and lower bounds of the solution's norm depend on the norm of $f$ via {\it nonlinear} power laws as it grows larger. Remark \hyperlink{R11}{1.1} will say more about how our study examines the nature of this nonlinearity.

In this article, we prove estimates in terms of  the modulus of continuity of $f$, denoted by $\om_f$. Our main result covers a wide range of scenarios even those where $f$ is discontinuous.  We show how the moduli of continuity of $\rD v$ and $\rD^2 v$ depend on $\om_f$ but does {\it not} depend on the modulus of {\it convexity} of $v$ itself -- the latter dependence is however required in \cite{XJW2} and \cite{AFCM1}.   Although one could argue that the result \cite[Cor. 2]{LC4} provided some information on what the modulus of convexity in turn depends on, a detailed and integrated study as we conduct in this article is still necessary. In fact, due to the lack of such details in \cite[Thm. 1(ii)]{XJW2}, their H\"older estimate appears to be of linear dependence which is shown to be not the case in \cite{AFCM1}.

The modulus of continuity of $f$ is defined as
\begin{equation*}
    \mc_{f}(q) := 
      \sup_{x, y\in\Om,\,|x-y|\le q} \big\{ |f(x) - f(y)|\big\}.
    \end{equation*} 
 It is non-decreasing for $q\in(0,\infty)$. If $q> \diam(\Om)$, then $\mc_{f}(q)=\mc_{f}\big(\diam(\Om)\big)$.  

Function $f$ is called {\it Dini continuous} if, for some $R>0$,
\[
 \int_0^{R} {\mc_{f}(q)}\,\frac{\rd q}{q} <\infty 
\]  
 which is a condition that is equivalent for any $R>0$. 
Burch in \cite{Burch} studied a large family of linear elliptic PDEs using Dini continuity. See also \cite{Kovats} by Kovats for early work on fully nonlinear elliptic PDE using $\om_f$ and note their assumption (*) is very close to H\"older continuity.
 
The main theorem is as follows where the assumptions on the domain $\Om$ and the lower bound of $f$ can be easily generalised using John's Lemma \ref{lem:J} and affine invariance  \eqref{T:D2} of MAE. Our theorem includes both the result allowing $ \int_0^{R} {\mc_{f}(q)}\,\frac{\rd q}{q} =\infty$ and the result assuming $ \int_0^{R} {\mc_{f}(q)}\,\frac{\rd q}{q} <\infty$. In fact, the proofs for both cases share many common parts and will be presented in a unified way.\medskip
 
 \noindent{\bf Notations}.  Let $\cB_{r}(x)$ denote the open ball of radius $r$ centered at $x$. Define $\cB_r:=\cB_{r}(0)$. Define
 \[
 \Omrho:=\big\{x\in\Om\,|\, \dist(x,\pa\Om)>\rho\big\}.
 \]
As usual, let $f\big|_{\Om_{\rho}}$ denote the restriction of $f$ onto $\Omrho$ so that $\mc_{f|_{\Om_{\rho}}}$ is defined accordingly.

Let $|\cdot|_{\cC^0(\Om)}$ denote the usual supremum norm and $|f|_{\cC^\ell(\Om)}:=|f|_{\cC^0(\Om)}+|\rD f|_{\cC^0(\Om)}+\cdots +|\rD^\ell f|_{\cC^0(\Om)}$. For $\alpha\in(0,1)$, seminorm $|\cdot|_{\cC^{\alpha}(\Om)}$ is defined as the maximal H\"older coefficient of the argument and $|f|_{\cC^{\ell,\alpha}(\Om)} :=|f|_{\cC^\ell(\Om)}+|\rD^\ell f|_{\cC^\alpha(\Om)}$  for integer $\ell\ge0$. 

\begin{theorem}[\bf Main Theorem]\label{mainresult}
There exist positive constants $\Cmc,\beta,\beta_1,\bar C$ that only depend on dimension $n$ so that the following is true. Consider a convex domain $\Om$ satisfying $\cB_1\subset\Om\subset \cB_n$ and a convex function $v \in \cC^2(\Om) \cap \cC^0(\overline{\Om})$ solving MAE \eqref{MA1}, \eqref{MA2} with $1\leq f\leq \fm$ in $\Om$ for  constant $\fm$. Fix $\rho>0$ and constant $\qin$ satisfying
\be\label{choose:first}
\qin\in(0,\rho] \qqua{ and } \mc_{f|_{\Om_{\rho-\qin}}}(\qin)\le \Cmc.
\ee
Let $\Qin$ denote \emph{various}  $(n,\fm)$-dependent positive powers of $\qin^{-1}$ times \emph{various} $(n,\fm,\rho)$-dependent positive constants. For any $x, x'\in \Omrho$, let $\ds :=|x-x'|>0$ and
\[
\cE_{\ds}:= \exp\Big(\beta\int_{\bar C\ds}^{\qin}{\mc_f(q)}\,\frac{\rd q}{q}\Big). 
\]
Then we have
\be\label{mainestimate:1}
\dfrac{\big|\rD v(x)-\rD v (x')\big|}\ds
\le
\begin{cases}    
\Qin \,\Big( \cE_{\ds}+\cE_{\ds}^{\beta_1}\mc_f(\Qin \, \cE_{\ds}\,\ds \,) \Big)&\text{ if  \; } \Qin\,\ds<1,\\
\Qin &\text{ if \; } \Qin\,\ds\ge 1.
\end{cases}
\ee
If further $f$ is Dini  continuous or equivalently $\cE_0<\infty$, then
 \be\label{C2:est}
\sup_{\Omrho} \big|\rD^2 v\big|\le\Qin\,\cE_0,
 \ee
 and 
\be\label{mainestimate}
\begin{aligned}
&\big|\rD^2 v(x)-\rD^2 v (x')\big| \le \\ 
&\quad  \begin{cases}    \Qin\,  \Big(\cE_{0}\,  \displaystyle\int_0^{ \Qin\, \cE_{\ds}\,\ds} {\mc_{f}(q)}\,\frac{\rd q}{q}+
         \big(1+     { \cE_{\ds}^2 }  \int_{ \bar C\ds}^{\qin}{\mc_f(q)}\,\frac{\rd q}{q^2}\,\big)\,\ds^{ }+\cE_{\ds}^{\beta_1}\mc_f\big(\Qin \, \cE_{\ds}\,\ds\,\big)\Big)&\text{ if  \; }  \Qin\,\ds < 1,\\[1mm]
 \Qin\,  \cE_{0} &\text{ if \; } \Qin\,\ds\ge1.
\end{cases}
\end{aligned}
\ee
Finally, the theorem still holds if we let $\mc_{f|_{\Om_{\rho-\qin}}}$ replace every occurrence of $\om_f$ in \eqref{mainestimate:1}\,--\,\eqref{mainestimate}, including those used in defining $\cE_{\ds}$ and $\cE_0$.
\end{theorem}

More information on constants $\Cmc$, $\beta$ and $\beta_1$ can be found respectively in \eqref{def:Cmc}, \eqref{def:beta} (also \S \ref{ss:beta}) and Lemma \ref{lem:ol}. They are independent of modulus of  convexity of $v$. Technically, Lemma \ref{lem:strict} shows this modulus of convexity depends on $n,\fm,\rho$. This information is ultimately linked to various positive powers of $\qin^{-1}$ that is used to define the $\Qin$ notation. Subscript ``$\inin$'' stands for ``initial'' because $\qin$ determines the size of the first section of the sequence in the iterative proof.

\begin{corollary}\label{Cor:Holder}
Consider a convex domain $\Om$ satisfying $\cB_1\subset\Om\subset \cB_n$ and a convex solution to  \eqref{MA1}, \eqref{MA2} in $\cC^{2,\alpha}(\Om) \cap \cC^0(\overline{\Om})$ for some $\alpha\in(0,1)$. Suppose $1\leq f \leq \fm$ in $\Om$ for  constant $\fm$. 
Fix $\rho>0$. Let $\wt C$ stand for \emph{various} $(n,\fm,\rho)$-dependent positive constants, $\Cht$ for \emph{various} $(n,\fm)$-dependent positive constants and $C$ for \emph{various} $n$-dependent positive constants. Then, with \(
\cH:=|f|_{\cC^{\alpha}(\Om)}
\), we have
\begin{align}
\label{cor:D2}\big|\rD^2 v\big|_{\cC^0(\Omrho)}&\le \wt C \,\Big(\big(C\cH\big)^{\Cht/\alpha}+1\Big),\\
\label{cor:Holder}\big|\rD^2 v\big|_{\cC^\alpha(\Omrho)}&\le \wt C \,\bigg(\dfrac{\big(C \cH\big)^{\Cht/\alpha}}{\alpha(1-\alpha)}+1\bigg).
\end{align} 
\end{corollary}
 \noindent{\bf Remark.} Thus, \eqref{cor:Holder} confirms and consolidates Theorems 1 and 2 of \cite{AFCM1}, at least  for $\alpha$ away from 0 and 1; it is slightly stronger for $\alpha\to0$ and slightly weaker for $\alpha\to1$. In view of their Theorem 3 regarding the {\it lower bound} of $\big|\rD^2 v\big|_{\cC^\alpha(\Omrho)}$, we see the sharpness of the power law \eqref{cor:Holder}. In fact, a relatively simple examination of their \S 4.5 also reveals the sharpness of power law \eqref{cor:D2}. \medskip
 
 \noindent{\bf Remark.}  Although their Theorem 3 does not address boundary conditions, a small change to their proof can indeed make \eqref{MA2} satisfied in a domain of bounded eccentricity: while they defined in \S 4.4 a family of examples $\{u_r\}_{r\to0}$  in domain $\cB_1$, we can instead define them in a fixed section of $u$ which is a convex function that they constructed to be independent of $r$. Also note by their Remark 4.1, their examples  have fixed (positive) upper and lower bounds of $f$.
\begin{proof}
Since $\om_f(q)\le \cH q^\alpha  $, we choose $\qin=\min\big\{(\Cmc/\cH)^{1/\alpha},\,\rho\big\}$ to have  \eqref{choose:first} satisfied.
Then, 
\be\label{cE:est}
\cE_{\ds}\le \cE_0\le \exp\big(\beta\alpha^{-1}\min\{\Cmc,\cH\rho^\alpha\}\big).
\ee 
Also, by the definition of $\Qin$ in Theorem \ref{mainresult}, 
\be\label{Qin:est}
\Qin\le \wt C\max\big\{(\cH/\Cmc)^{\Cht/\alpha},\,\rho^{-\Cht}\big\}.
\ee
Then, by  \eqref{C2:est}, we prove \eqref{cor:D2}.
 
Next, for any $x,x'\in\Omrho$, we apply \eqref{mainestimate}  to have
\[
\frac{\big|\rD^2 v(x)-\rD^2 v (x')\big|}{\ds^{\,\alpha}}\le
\begin{cases}
\Qin\,\cH\,\Big(\cE_{0}^{1+\alpha}/\alpha+\cE_{\ds}^2/(1-\alpha)+\cE_{\ds}^{\beta_1+\alpha}\Big)+\Qin&\text{ if  \; }  \Qin\,\ds < 1,\\[1mm]
 \Qin\,  \cE_{0} /\ds^{\,\alpha}&\text{ if \; } \Qin\,\ds\ge1.
\end{cases}
\]
In view of \eqref{cE:est} and \eqref{Qin:est}, 
 we  prove \eqref{cor:Holder}.
\end{proof} 

\noindent\hypertarget{R11}{{\bf Remark 1.1.}}
This proof reveals the intimate link of $\qin,\Qin$ to the nonlinear dependence of H\"older estimates.  First, \eqref{cE:est} implies a uniform upper bound on $\cE_{\ds}$ as $\cH$ grows larger. We will see in \S \ref{sec:sections} and \S \ref{sec:iterativeproof}  that such uniformity indicates the Hessian matrix $\rD^2v$ is under good control in neighbourhoods of size below $\qin$; in the words of \cite[p. 3811]{AFCM1}, this is where ``linearity kicks in''. At large scales, nonlinear effect dominates. Take the above proof for example: as $\cH$ grows, $\qin$ shrinks and so does the size of a ``linear neighbourhood"; meanwhile, \eqref{Qin:est} on $\Qin$ grows nonlinearly in $\cH$.

We now introduce generalised (weak) solutions named after Alexandrov.  The definition is based on the notion of Monge-Amp\`ere measure: for any convex function $u(x)$ defined on $\Om\subset \bR^n$,  its sub-differential $\pa u$ maps any Borel set  $ \Om'\subset\Om$ to a measurable set $\pa u(\Om')$ whose Lebesgue measure is called the Monge-Amp\`ere measure of $\Om'$ associated to $u$. Notation ${\cM}u(\Om'):=\big|\pa u(\Om')\big|$ is used. If $u\in \cC^2(\Om)$, apparently ${\cM}u$ conincides with $\det \mathrm D^2u(x)\mathrm dx$.

\begin{corollary}\label{cor:1:3}
Let $\Cmc,\beta,\beta_1$ denote the same constants  as in Theorem \ref{mainresult}. 
For a convex domain $\Om$ satisfying $\cB_1\subset\Om\subset \cB_n$ and  $f\in L^\infty(\Om)$ satisfying $f(x)\ge1$ a.e. in $\Om$,
define
\[
   \mc_{f}(q) := 
    \textnormal{ess\,sup}_{x\in\Om} \Big\{ \textnormal{ess\,sup}_{y\in\Om\atop|x-y|<q} \big|f(x) - f(y)\big|\Big\}.
 \]
  If $\om_f(0+)<\min\big\{\Cmc,1/\beta_2\big\}$ for $\beta_2:=\beta\max\{1,\beta_1\}$, then there exists a unique convex, generalised (weak) solution {\it a la} Alexandrov to the MAE
\be\label{cM:v}
\cM v=f^n\rd x\qqu{with}v\big|_{\pa\Om}=0,
\ee
with regularity $v\in \cC^{1,1-\gamma}_{\textnormal{loc}}({\Om})\cap\cC^0(\overline\Om)$ for any constant $\gamma$ satisfying
\be\label{cond:g}
\beta_2\,\om_f(0+)<\gamma <1.
\ee
The equality $\gamma=\beta_2\,\om_f(0+)$ can be attained if further $\int_{0}^{1}\left(\mc_f(q)-\om_f(0+)\right)\,\frac{\rd q}{q}<\infty.$ 
Finally, there exists an $(n,|f|_{L^\infty(\Om)},\rho,\om_f,\gamma)$-dependent upper bound for $|\rD v|_{\cC^{1-\gamma}(\Omrho)}$.
\end{corollary} 

\noindent{\bf Remark.}
This result may seem weaker than that of \cite{C:C1alpha} which does not require the smallness of $\mc_f(0+)$. However, we provide explicit dependence of the H\"older exponent $(1-\gamma)$ on $\mc_f$ and shows that $(1-\gamma)$ can be approaching 1 if the size of discontinuities in $f$, measured by $\mc_f(0+)$, approaches 0. Another route to making $(1-\gamma)$ approach $1$ is to show $v$ is in Sobolev space $W^{2,p}$ for arbitrarily large but finite $p$ (\cite{LC5}).  

The proof is postponed to Appendix \ref{app:cor:1:3}. The same mollification and compactness argument there can be used to prove many more existence and regularity results as consequences of Theorem \ref{mainresult}. We only mention two examples that involve logarithmic refinement of continuity function classes. The sketch proofs of them are also in Appendix \ref{app:cor:1:3}.

{\bf Example 1}:  under the same setup as Theorem \ref{mainresult}, if
\[
\sigma_0:=\beta\,\limsup_{q\to0+}\mc_f(q)\,|\ln q|<\infty,
\]
then for $\sigma>\max\{\sigma_0,\,\sigma_0\beta_1-1\}$,  there exists an $(n,\fm,\rho,\om_f,\sigma)$-dependent upper bound for $\dfrac{\big|\rD v(x)-\rD v(x')\big|}{\ds\,|\ln\ds|^{\sigma}}$ with $\ds\le\frac12$.\medskip

{\bf Example 2}: under the same setup as Theorem \ref{mainresult}, if there exists constant $a>1$ so that
\[
\limsup_{q\to0+}\om_f(q)\,|\ln q|^a<\infty,
\]
then there exists an $(n,\fm,\rho,\om_f,a)$-dependent upper bound for $\dfrac{\big|\rD^2 v(x)-\rD^2 v(x')\big|}{|\ln\ds|^{a}}$ with $\ds\le\frac12$.\medskip

The paper is organised as follows. In \S \ref{sec:well-known}, we collect some well-known results regarding MAE with constant or non-constant right-hand side. We also introduce normalised and quasi-normalised domains. In \S \ref{sec:sections}, we introduce sections of MAE solution and a notation $\cF_T$ that denotes a group action of the affine group on functions. We explain what it means for the MAE to be invariant under this group action -- also known as {\it affine invariance}. Using these concepts, in Lemma \ref{Stepk}, we study MAE  with right-hand side sufficiently close to unit in the supremum norm. Building on this, we employ  in \S \ref{sec:iterativeproof} an iterative procedure to construct a sequence of shrinking sections of an MAE solution and obtain estimates that are linked to their ``eccentricities''. This is still a perturbative argument just like in \S \ref{sec:sections} because the right-hand side of MAE, when restricted to the first and largest section in the sequence, must be close to unit.  In order to kick-start the iteration for more general right-hand side, in \S \ref{sec:first:step}, we  prove the existence of small enough (first) section about any point inside the domain, and also obtain estimates on the geometry of this section, which will be ultimately linked to the scaling quantity $\qin$ in \S \ref{sec:mainproof} where we complete the proof of the Main Theorem.\medskip

\noindent{\bf Remark.} If a global $\cC^{1,1}(\overline\Om)$ bound on $v$ was already available, then our proofs would potentially be much simpler. We are not aware of such literature when one only assumes  $\cC^0$ regularity of $f$ and does not assume strict or uniform convexity of  $\Om$, or certain regularity of $\pa\Om$.  Consult, e.g., \cite[\S 17]{GTB} for global estimates and more detailed reference list on this topic; also see \cite[Remark 3.6]{AFCM1} for references of early work.
\medskip

\section*{\bf Acknowledgement}
Cheng is supported by the Leverhulme Trust (Award No. RPG-2017-098) and the EPSRC (Grant No. EP/R029628/1). O'Neill is supported by an EPSRC PhD studentship.
\medskip

\section{\bf Well-known Results on Monge-Amp\`ere Equation}\label{sec:well-known}

Geometric techniques are crucial in the study of MAE. This is particularly due to its ``affine invariant'' nature, a term we use loosely to refer to how the Hessian of a function changes under affine transform on the independent variable, which makes  the change of the determinant of the Hessian to be simply a constant scaling. 

First, we state a fundamental result on the geometry of convex sets. 

\begin{lemma}[John's Lemma (\cite{JohnF})]\label{lem:J}
If $\Om \subset \bR^n$ is a bounded, open convex set with nonempty interior and let $E$ be the ellipsoid of minimum volume containing $\Om$. Then
\[
n^{-1} E \subset \Om \subset E
\]
where  $n^{-1}E$ denotes the dilation of $E$ with respect to its centre.
\end{lemma} 

 For such ellipsoid $E$, there exists an affine transform $Tx=A(x-\xi)$ for some $\xi\in\bR^n$ and  real, invertible matrix $A$ that maps $n^{-1}E$ to the unit ball, which also maps $\Om$ to a so-called {\it normalised} set $T\Om$. Apparently, matrix $A$ can be chosen to be symmetric positive definite without ``rotation'' (otherwise apply polar factorisation on $A$). Although many results in literature are stated under the assumption of a normalised domain, we remark that the minimality of $E$ is not necessary for their proofs to work. The essential assumption is that the domain can be affine transformed to one that contains $\cB_1$ and is contained in $\cB_n$.  We shall call such domain {\it quasi-normalised}. 
  
Next onto properties of convex function $u$. A basic fact is that it is  locally uniformly Lipschitz (\cite[Cor. A.23]{AFB2}). In particular, if a convex function $u\le 0$ on $\Om$ and it is differentiable at $x\in\Om$, then an elementary calculation (\cite[Lem. 3.2.1]{CG1}) gives
\be\label{Lipbound}
|\rD u(x)|\le \frac{-u(x)}{\dist(x,\pa\Om)}.
\ee

There are a couple of maximum principles:  the Alexandrov's maximum principle (\cite[Thm. 1.4.2]{CG1}, \cite[Thm. 2.8]{AFB2}) for convex functions in $\cC^0(\overline\Om)$ and the Alexandrov-Bakelman-Pucci's maximum principle (\cite[Lem. 9.2]{GTB}, \cite[Thm. 1.4.5]{CG1}) for general functions in $\cC^0(\overline\Om)$. 
For our purposes, it suffices to adapt the former one to $\cC^2$ functions.
\begin{lemma} \label{lem:max}
Consider a bounded and convex set $\Om \subset \bR^n$  and a convex function $u \in \cC^2(\Om)\cap \cC^0(\overline{\Om})$ with $u = 0$ on $\pa \Om$. Then, for constant $C_{\!\mathsf A}$,
\be\label{Alexandrov}
|u(x)|^n \leq C_{\!\mathsf A}\, \diam(\Om)^{n-1}\dist(x,\pa \Om)\int_\Om \det\mathrm{D}^2u(y){\rd}y,\qquad\forall\,x \in \Om.
\ee
\end{lemma}

A useful consequence is that points at which the solution takes different values are well separated.

\begin{corollary}\label{cor:sep}
Under the same assumptions and notations as Lemma \ref{lem:max},  for any $x_1,x_2\in\overline{\Om}$,
\[
\dist(x_1,x_2)\geq \frac{|u(x_1)-u(x_2)|^n}{ C_{\!\mathsf A}\,\diam(\Om)^{n-1}\int_\Om \det\rD^2u(y){\rd}y}\,,
\]
which naturally implies the \emph{global} $\cC^\frac1n(\overline\Om)$ bound on $u$ if $\det\rD^2u$ is uniformly bounded in $\Om$.
\end{corollary}
\begin{proof}
Suppose $u(x_1)>u(x_2)$. Apply Lemma \ref{lem:max} to the function $u(x)-u(x_1)$ over the domain
$
\big\{x\in\Om\,\big|\, u(x)<u(x_1)\big\}.
$
\end{proof}

Another type of well-known results is the comparison principle (e.g., \cite[Thm. 1.4.6]{CG1}, \cite[Thm. 2.10]{AFB2}).  We again adapt it for $\cC^2$ functions in the next lemma.

\begin{lemma}\label{lem:comp}
Consider a bounded and convex set $\Om \subset \bR^n$  and convex functions $u,v \in \cC^2(\Om)\cap \cC^0(\overline{\Om})$. If $\det\mathrm D^2 u(x)\le \det \mathrm D^2 v(x)$ in $\Om$ and $u \ge v$ on $\pa \Om$, then $u(x)\ge v(x)$ in $\Om$.
\end{lemma}
 
\noindent {\bf Definition}. For a bounded and strictly convex domain $\Om$, define $w=\solM{\Om}$ to be the strictly convex solution to the boundary value problem
\begin{equation}\label{def:solM}
  \det \mathrm{D}^2w = 1 \qu{in}\Om \qqu{with} w\big|_{\pa\Om}=0.
\end{equation}   
For the existence and uniqueness of the (convex) Alexandrov solution $w\in \cC^0(\overline\Om)$ to \eqref{def:solM}, we refer to \cite[Thm. 1.6.2]{CG1} and \cite[Thm. 2.13]{AFB2} with the latter result actually {\it not} requiring {strict} or uniform convexity of $\Om$. The strict convexity of $w$ inside $\Om$ is due to \eqref{mm:strict}. Also, see, e.g.,  \cite[Thm. 3.10]{AFB2} for the proof of $w\in\cC^\infty(\Om)$. \medskip

\noindent{\bf Notations}. 
All versions of $w$ denote convex solutions to the MAE with unit right-hand side.

 \begin{corollary}\label{cor:est:com}
 Consider a strictly convex domain $\Om\subset\bR^n$ satisfying $\cB_{r_1}(\xi_1)\subset \Om\subset \cB_{r_2}(\xi_2)$ for positive constants $r_1, r_2$ and $\xi_1,\xi_2\in\bR^n$. Consider a strictly convex solution $u\in\cC^2(\Om)\cap\cC^0(\overline\Om)$ to
 \begin{align*}
     \detn \mathrm{D}^2u &= f \qu{    in    }\Om , \\
      u &=0 \qu{    on    }\pa \Om .      
\end{align*} 
Suppose $\sup_\Om| f- 1 |<1$ and let $w:=\solM{\Om}$. Then, we have
\begin{align}\label{est:wm}
 -\frac {r_2^2}2\le 
 \inf_{\Om}w&\le -\frac {r_1^{2}}2,\\
 \label{comp12}
\sup_\Om| u- w | &\le  \frac{r_2^2}2\sup_\Om| f- 1 |,\\[1mm]
\label{vw:mm:comp}\big|\inf_{\Om}u- \inf_{\Om}w \big|&\le\sup_\Om|u-w|.
\end{align}
 \end{corollary}

\begin{proof}
 Let 
 \[
 w(z )=\textstyle\inf_{\Om}w\qqu{for some}z\in\Om.
 \] By applying Lemma \ref{lem:comp} to  $\frac {1}2 (|x-\xi_1|^2-r_1^2)$ and $w(x)$, we  find $\inf_{\Om}w\le w(\xi_1)\le -\frac {1}2 r_1^2$, hence proving the upper bound of \eqref{est:wm}.
 On the other hand, applying Lemma \ref{lem:comp} to $w(x)$ and $\frac {1}2  (|x-\xi_2|^2-r_2^2)$, we  find    $w(z )\ge\frac {1}2 (|z -\xi_2|^2-r_2^2)\ge -\frac {1}2  r_2^2$, hence proving the lower bound of \eqref{est:wm}. 
 
Next, with $\delta:=\sup_\Om| f- 1 |$, we have 
\[
\detn \mathrm{D}^2(1-\delta)w \leq \detn \mathrm{D}^2u \leq \detn\mathrm{D}^2(1+\delta)w \qqu{in} \Om.
\]
Then  by comparison principle Lemma \ref{lem:comp} we have 
\be\label{comp:int}
(1+\delta)w \leq u   \leq (1-\delta ) w \qqu{in} \Om.
\ee
 In view of \eqref{est:wm}, we prove \eqref{comp12}.

Finally,  let $u(y )=\inf_{\Om}u$ for some $y\in\Om$. Then, we bound $u(y) - w(z)$ from above and from below, respectively, using $u(y)\le u(z)$ and $-w(z)\ge -w(y)$. Thus we prove \eqref{vw:mm:comp}.
\end{proof}

Now we move on to {\it strict} convexity of solutions to MAE. Caffarelli proved a fundamental result in \cite[Thm. 1]{LC4} in terms of extremal points of the set over which $u$ is affine. It led to fruitful consequences in strict convexity (which prohibits $u$ from being affine) and solution regularity. By, e.g., \cite[Cor. 5.2.2]{CG1}, \cite[Cor. 4.11]{AFB2}, for constants $0<f_{\min}\le \fm<\infty$ and a convex function $u\in\cC^0(\overline\Om)$ with $\Om$ being a convex and bounded domain,
\be\label{mm:strict}
\begin{aligned}
&f_{\min}\mathrm dx\le {\mathcal M}u\le  \fm\mathrm dx \;\;\text{ \;\; and \;\; } \;\; u\big|_{\pa\Om}=0\\
&\text{imply {\bf strict} convexity of $u$ in the interior of }\Om.
\end{aligned}
\ee
This  means all (classical and generalised)  solutions to MAE in this article are strictly convex in the interior of their respective domains, a fact that we will take for granted from now on. Moreover, in any  occurrence of $\solM{\cdot}$ in this article, the strict convexity requirement for its definition  \eqref{def:solM} is always satisfied because we will make sure the domain for every instance of $\solM{\cdot}$ is always the section of another convex MAE solution in the {\it interior} of its domain, hence being strictly convex. 

To finish this section, we state the following results on Monge-Amp\`ere equation with constant right-hand side and sketch ideas behind their proofs.

\begin{lemma} \label{C4Bound}
Let $\Om\subset \bR^n$ be  strictly convex and satisfy $ \cB_1\subset\Om \subset \cB_n$. Let $w:=\solM{\Om}$ and 
\[
\Ompa{s}:=\big\{x\in\Om\,\big|\,w(x)<-s\big\}.
\]
Consider any $s\in(0,\frac12 n^{-2})$. 
If $\Ompa{2s}\ne\emptyset$, then  there exists positive constant $\tilde c_2(n,s)$ so that 
\be\label{Pog:real}
\tilde c_2^{1-n} I\le  \rD^2w\le {\tilde c_2 }I\qqu{in}\Ompa{2s}. 
\ee
If $\Ompa{4s}\ne\emptyset$, then there exist  constant  $\tilde c_3(n,s)$ so that 
\be \label{est:C3a}
   |\rD^3 w|_{\cC^{0,1/2}(\Ompa{4s})}   \le  \tilde c_3.
\ee
\end{lemma}
\begin{proof} The existence and $\cC^\infty(\Om)$ regularity of $w$ is discussed below \eqref{def:solM}.
 
By  the Lipschitz bound \eqref{Lipbound} and maximum principle Lemma \ref{lem:max}, we have 
\[
\textstyle\sup_{\Ompa{s}}|\rD w|\le {\tilde c_1}/\dist\big(\Ompa{s},\pa\Om\big)^{n-1\over n}
\]
 for  constant $\tilde c_1$. Meanwhile, the separation estimate in Corollary \ref{cor:sep} also gives a lower bound on $\dist\big(\Ompa{s},\pa\Om\big)$. Thus
\be\label{est:Dw:s}
\textstyle\sup_{\Ompa{s}}|\rD w|\le \tilde c_1'(n ,s)
\ee
for  constant $ c_1'(n ,s)$. Now consider 
\[
w_{\langle s \rangle}(x):=w(x)+s\qu{on}\Ompa{s}.
\]  By Pogorelov's interior estimates (\cite{P:est}, \cite[Lem. 4.1.1]{CG1}), there exists a scalar function $G(\cdot,\cdot)$ that is monotonically increasing with the second argument so that
\[
|w_{\langle s \rangle}|\,D^2w_{\langle s \rangle}\le G\big(n,\textstyle\sup_{\Ompa{s}}|\rD w_{\langle s \rangle}|\big)I \qqu{ in}\Ompa{s}.
\] 
 Combining it with \eqref{est:Dw:s} and the fact that $|w_{\langle s \rangle}|\ge s$ in $\Ompa{2s}$ proves the upper bound of \eqref{Pog:real}. Since  all eigenvalues of $\mathrm{D}^2w$, counting multiplicity, multiply to 1, we prove the lower bound of \eqref{Pog:real}. 

Now that we have shown $w$ is uniformly elliptic on $\Ompa{2s}$, we then can show the uniform ellipticity for the linearised version of $\det\rD^2w=1$ as well as for the PDE of a given component of $\mathrm{D}w$ that results from taking one derivative of $\det\rD^2w=1$. Note that Corollary \ref{cor:sep}  gives  lower bounds on $\dist(\Ompa{3s},\pa\Ompa{2s})$ and $\dist(\Ompa{4s},\pa\Ompa{3s})$. 
This means we can   apply  Evans-Krylov's H\"older interior estimates on  $\rD^2 w$; see, e.g.,  \cite[eq. (17.63) - (17.65), (17.32), (17.33), especially (17.41)]{GTB}. Then, \eqref{est:C3a} follows from Schauder's interior estimates (e.g., \cite[Thm 6.2]{GTB}).
\end{proof}

The next lemma provides  estimates on the difference of two solutions to MAE with constant right-hand sides. It requires information of two solutions in a shared domain, but does not require any information at the boundary.

\begin{lemma} \label{lem:ol}
For positive constants $r,b,b',\delta,\lambda, \Lambda$ and $\alpha \in(0,1)$, suppose  $w, w'\in\cC^{2,\alpha}({\Oms})$ satisfy
 \[
 |bw-b'w'|\le\delta, \qquad\quad
\det \mathrm{D}^2w=\det \mathrm{D}^2w'=1
\]
and matrix inequalities $\lambda I\le \rD^2 w,\rD^2 w'\le \Lambda I$ in $\Oms$.
Let $C_{\alpha }$ be an upper bound of the $\cC^{\alpha }(\Oms)$ seminorms of $\rD^2 w$ and  $\rD^2w'$. 
Then,   there exist positive constants $\beta_1,\wt\beta_1,C$ that only depend on $n,\alpha$ with $\beta_1\ge\wt\beta_1$ so that
\begin{equation} 
\label{C2Est}
\begin{aligned}
& r_0^{\ell}\,\big|\mathrm{D}^\ell (w-w')\big|_{\cC^0(\cB_{r-r_0})}+r_0^{\ell+\alpha}\,\big|\mathrm{D}^\ell (w-w')\big|_{\cC^\alpha(\cB_{r-r_0})}\\
\le \,& C\,\lambda^{-\widetilde\beta_1}\Lambda^{\beta_1-\widetilde\beta_1}\,  \big(  1 + r^{\alpha }C_{\alpha }\big)\,\frac{\delta+r^{2 }|b-b'|}{b+b'}\,, 
\qquad\qquad\forall\,r_0\in(0,r) \text{ \;  and \; }\ell=2,3.
\end{aligned}
\end{equation}
\end{lemma} 
\begin{proof} Define
\[\begin{aligned}
u(x)&:=(bw-b'w')(x)-\frac{b-b'}2\,\big((w+w')(x)-(w+w')(0)+x\cdot (w+w')(0)\big).
\end{aligned}
\]
By the assumed bound on $(bw-b'w')$ and Taylor expansion on $(w+w')$ about $0$ together with the assumed $\cC^2$ upper bound on $w,w'$, we have
\begin{equation}\label{u:sigma}
|u|_{\cC^0(\cB_r)}\le\delta+\Lambda\,r^2|b-b'|\,.
\end{equation}
Next, in the identity, $\det \mathrm{D}^2w-\det \mathrm{D}^2w'=\displaystyle\int_0^1{{\rd}\over {\rd}t}\det\rD^2 w_t\,{\rd}t$ for $w_t:=t w+(1-t)w'$, we apply Jacobi's formula on the $t$ derivative, followed by replacing $\rD^2 (w-w')={2\over b+b'}\rD^2 u$, to find,  
\be\label{w:theta}
\textnormal{trace}\left(\rD^2 u\int_0^1\text{adj}(\rD^2w_t)\,{\rd}t\right)=0
\ee
where ``adj'' denotes matrix adjugate so that $\text{adj}(\rD^2w_t)=(\rD^2w_t)^{-1}\det\rD^2w_t$. Since $\rD^2w_t$ is a convex combination of $\rD^2 w$ and $\rD^2 w'$, we apply the Minskowski determinant theorem to obtain a lower bound for $\det\rD^2w_t$ and also use $\Lambda^{n}$ as the upper bound for $\det\rD^2w_t$ to have
\[
\Lambda^{-1} I\le \text{adj}(\rD^2w_t)\le \lambda^{-1}\Lambda^{n} I
 \]
so we can view \eqref{w:theta} as a uniformly elliptic, 2nd order {\it linear} PDE for $u$. We then prove the $\ell=2$ case of \eqref{C2Est} via combining \eqref{u:sigma} and Schauder's interior estimates applied on \eqref{w:theta}, a process that will incur a factor at the order of the distance-weighted $\cC^{0,\alpha }(\Om)$ norm of  the coefficients of $\rD^2 u$ in \eqref{w:theta} for which there exists an upper bound at the order of $\big|\text{adj}(\rD^2w_t)\big|_{\cC^0}+r^{\alpha }\big|\text{adj}(\rD^2w_t)\big|_{\cC^\alpha}$. Here, an entry of $\text{adj}(\rD^2w_t)$ equals a sum of products of  entries of $\rD^2w_t$. Also,  consult the proofs of \cite[Lem. 6.1, Thm. 6.2]{GTB} to see that the final estimates depend linearly on $\big|\text{adj}(\rD^2w_t)\big|_{\cC^\alpha}$ and also depend  on $\lambda$, $\Lambda$ via power laws (which has been simplified in our case since $\lambda\le 1\le\Lambda$).
    
The $\ell=3$ case of \eqref{C2Est} can be proven by mimicking  the above steps, starting with taking an arbitrary directional derivative of  \eqref{w:theta}.
\end{proof}

\noindent {\bf Remark.}  By the definition of $u$ and the assumed $\cC^2$ upper bound on $w,w'$, we can also bound
 \be\label{bw:ww}
\big|\rD^2(bw-b'w')(x)\big|\le  \tfrac12(b+b')\big|\rD^2(w-w')(x)\big|+\Lambda|b-b'|\,,\qquad\forall\,x\in\Oms\,.
\ee

\noindent{\bf Convention on constants}. From now on, constants are always strictly positive and, unless  their dependence on other parameters is explicitly given,  constants only depend on dimension $n$. 

\section{\bf Sections of Solutions to Monge-Amp\`ere Equation}\label{sec:sections}
The use of {\it sections} in studying the solution regularity of MAE has seen success since  Caffarelli's \cite{LC5,  C:C1alpha, LC4}. Here, we define it only for $\cC^1$ function $v$, and in the following form as a mapping from a triplet of arguments to a subdomain,
\begin{equation*}
 \cS(v,h,\xs) :=   \{ z \in \Om: v(z) < v(\xs) + (z-\xs) \cdot \nabla v(\xs) + h\}.
\end{equation*}
 We say this section is about point $\xs$. Let  $\cScmd{v,h}:= \cS(v,h,0)$.

 To describe the effects of affine transforms, we introduce the following notations. For any member of the affine group on $\bR^n$, i.e.,  an  invertible affine transform $T$,  define 
 \[
 \det T:= {\big|T\cB_1\big|}/{|\cB_1|},
 \]
 so that if $Tx=A(x-\xi)$ for matrix $A$ and some point $\xi\in\bR^n$, then $\det T=\det A$. 
 
 For function $v$, define
 \[
( \cF_T\,v)(y):=(\detns T) \,v\big(T^{-1}y\big).
 \]
 Then $\cF_T$ is a left group action of the affine group on functions, namely, it satisfies
 \begin{align*}
\cF_{id}\,v &=v\qqua{and}\cF_{T_1}(\cF_{T_2}\, v) =\cF_{T_1 T_2}\,v\,.
 \end{align*}
 We also have the following  ``commutation property'' between $T,\cF_T$ and the sections 
 \be\label{T:sn}
T\,\cS(v,h,\xs)=\cS\big(\cF_T\,v,(\detns T) \,h,T\xs\big),
 \ee
 which is straightforward to prove when $v$ is $\cC^1$. 

 Next,  we define the norm for a square matrix $A$ to be induced from the vector $2$-norm, namely
 \[
 \mat{A}:=\sup_{0\ne x\in \mathbb R^n}{|Ax|\over|x|},
 \]
and define the norm of affine transform  $T$ as
\[
\mat{T}:=\sup_{x_1\ne x_2}\frac{|Tx_1-Tx_2|}{|x_1-x_2|},
\]
so that if $Tx=A(x-\xi)$ then $\mat{T}=\mat{A}$.

 Derivatives under affine transforms are governed by the following elementary properties,
  \begin{align}
\label{T:Dk:al} 
\left.\begin{aligned}
|\rD^\ell(\cF_T\,v)|_{\cC^{\alpha}(T\,\Om)}&\le |\rD^\ell v|_{\cC^{\alpha}(\Om)}\mat{T^{-1}}^{\ell+\alpha}\detns T,\\
 &\qquad\forall \;v\in \cC^{\ell,\alpha}(\Om),\;\; \alpha\in[0,1]\text{ \;and integer \, }\ell\ge0,\\
  \end{aligned}\;\;\right\}\\  \label{T:D2}
	\det\rD^2_y(\cF_T\,v)(y)=\det\rD^2_xv(x),
  \qquad\forall \; v\in\cC^2(\Om),\;\; x\in\Om\text{\; and \;}y=Tx.
 \end{align}
We will refer to \eqref{T:D2} as the ``affine invariance'' of MAE from here on.

Next, for invertible affine transform $T$ and matrix $A$, we introduce their rescaled versions  as
\begin{align*}
\br{T}&:=\big({ \detni T}\big)T\qquad\text{and}\qquad\br{A}:=\big({ \detni A}\big)A,
\end{align*}
so that $\det\br{T}=1=\det\br{A}$. As an attempt to increase readability, we use $\mt{\;}$ to denote the matrix/transform norm when the argument is known to have unit determinant, namely, 
\[
\mt{\br{T}}:=\mat{\br{T}}\qquad\text{and}\qquad\mt{\br{A}}:=\mat{\br{A}}.
\]
 Note $\br{(A^{-1})}=(\br{A})^{-1}$ and  $\mt{\br{(T^{-1})}}=\mt{(\br{T})^{-1}}$.

The product ${\mt{\br{A}}}\cdot{\mt{\br{A}^{-1}}}$ equals the condition number of $\br{A}$ induced by the vector $2$-norm, and equals the eccentricity of the ellipsoid $A\cB_1$. Moreover, for any real, invertible matrix $A$, by polar factorisation, there exists an invertible, real, {\it symmetric} matrix $A_1$ with $\det A_1=1$ so that $\mt{\br{A}}=\mt{{A}_1}$ and $\mt{\br{A}^{-1}} =\mt{{A}_1^{-1}}$. Since $ {A}_1$ has $n$ real eigenvalues (counting multiplicity) whose product equals unit, and among them, the smallest absolute value equals $\mt{{A}^{-1}_1}^{-1}$ and the greatest absolute value equals $\mt{{A}_1}$, we have
$  \mt{{A}_1^{-1}} \le  \mt{{A}_1}^{n-1}$ and ${\mt{{A}_1}\le \mt{{A}^{-1}}}^{n-1}$, hence
\be\label{sc:A:A1}
{\mt{\br{A}}^{1\over n-1}\le \mt{\br{A}^{-1}}} \le  \mt{\br{A}}^{n-1}\,,\qquad\;\;\forall\, \text{real, invertible matrix }A.
\ee  

In the rescaled notations, the transform estimates \eqref{T:Dk:al} amount to
\be\label{T:Dk:al:rs} 
 \left. \begin{aligned}
|\rD^\ell(\cF_T\,v)|_{\cC^{\alpha}(T\,\Om)}&\le |\rD^\ell v|_{\cC^{\alpha}(\Om)}\mt{\br{T}^{-1}}^{\ell+\alpha} \big(\detn T\big)^{2-\ell-\alpha} ,\\&\quad\forall \; v\in\cC^{\ell,\alpha}(\Om), \; \; \alpha\in[0,1]\text{ \; and integer \, }\ell\ge0 .
 \end{aligned}\quad\right\}
\ee

The next lemma is on MAE with right-hand side sufficiently \emph{close to unit}, and will serve different technical purposes in the proofs of Lemmas \ref{lem:iter} and \ref{lem:Dini}. We are motivated by \cite[Lem. 7, Cor. 2, 3]{LC5} while endeavouring to provide enough details that will be useful in proving the Main Theorem. For example, for the iterative argument that we will employ in \S \ref{sec:iterativeproof}, estimates \eqref{est:D2w1w0} are concerned with   $\cF_T\,w^\sharp$ which will effectively plays the role of $w$ in the next step of iteration, and the right-hand side of \eqref{est:D2w1w0} shows the distance between identity and its Hessian -- crucially the square root of this Hessian is used to define the next affine transform in the iterative proof of \S \ref{sec:iterativeproof}.  

\begin{lemma}\label{Stepk}
 Suppose function $u$ is strictly convex and $\cC^2$ in a closed section  $\overline{\cScmd{u,L}}$ for $L>0$, and suppose $\cB_1(\xi) \subset \cScmd{u,L}\subset \cB_n(\xi)$ for some $\xi\in\bR^n$.
Also suppose
\be\label{v:cond}
\detn\rD^2u=g\qqua{with} \rD u(0)=\vec0,\quad u\big|_{\pa\cScmd{u,L}}=0,\quad g(0)=1,
\ee
 and
\[
 \delta:=\textstyle\sup_{\cScmd{u,L}}|g-1|\le\frac15n^{-2}.
\]
Then there exist  positive constants $c_2,C_2,C_3$  so that $w:=\mathsf{Sol}(\cScmd{u,L})$ satisfies
\begin{align}
\label{vw:sn} \big| u - w \big| \le  \tfrac12n^2\delta&\qqu{in}\cScmd{u,L}\,,\\
\label{Pog0:3i} 
c_2\,I \leq \mathrm{D}^2w \leq C_2\,I &\qqu{in}\cScmd{u,L/2}\,,\\
\label{Pog0:3ii} |\rD^2 w|_{\cC^{\alpha}}  \leq C_2^{1-\alpha}C_3^{\alpha}&\qqu{in}\cScmd{u,L/2}\,,\quad\forall\,\alpha\in[0,1]\\
\label{gradest}
|\mathrm{D}w(0)| \leq&\, nC_2\big(2  c_2^{-1}\big)^{\frac12}\delta^{\frac12}.
\end{align}

 Further, there exist positive  constants $C_4\le 1$ and $ \hc\le\frac15$ so that if  
\begin{align}\label{cond:scaling1}
0&<h\le  \hc\,,\\
\label{cond:scaling2}0&\le \delta \le \min\left\{C_4\,h\,,\,\tfrac15n^{-2}\right\},
\end{align}
 then $\cScmd{u,h}\subset\cScmd{u,L/2}$ and the following hold.
 \begin{itemize}
 \item[(i)] The affine transform
 \[
 Tx:= h^{-{1\over2}}\sqrt{\mathrm{D}^2w(0)}\,x
\]
acting on section $\cScmd{u,h}$ satisfies
 \begin{align}
 \label{n:d:1} &\cB_1  \subset  \cB_{r_-}\subset \,T\cScmd{u,h}  \subset \cB_{r_+}\subset \cB_2\\
 &\qquad\qqu{ with} \label{n:d:2}
  \Big|\frac12 r_\pm^2-1\Big| \le  O(\delta/h)+C_3\, O\big(\sqrt h\big),
 \end{align} 
 where the hidden constant coefficients in the $O(\,)$ notations only depend on $n,c_2,C_2,C_4$.\medskip
 \item[(ii)] For $w^\sharp:=\solM{\cScmd{u,h}}$,  there exist  positive constants $C_5,C_6$ so that
\begin{align}
  \label{est:D2w1w0}\big|\rD^2\big(\cF_T\,w^\sharp- \cF_T\,w\big)(0)\big|=\big|\rD^2(\cF_T\, w^\sharp)(0)- I\big|   
&\le C_5\,\delta/h ,\\[1mm]
  \label{est:D3w1w0}\textstyle\sup_{\cB_{1/4}}\big|\rD^3\big(\cF_T\, w^\sharp- \cF_T\,w\big)\big|& 
\le C_6\,\delta/h.
\end{align}
\end{itemize}
\end{lemma}

\noindent{\bf Remark.} Note that \eqref{n:d:1} offers a simple way to quasi-normalise the smaller section $\cScmd{u,h}$.  Also, under a further assumption $\cB_{r_-'} \subset \cScmd{u,L} \subset \cB_{r_+'}$,  we can replace $C_2$ by  $1+C_2(r_+'-r_-')$  and replace $C_3$ by $C_3(r_+'-r_-')$. For simplicity, we will not make such refinement in this article. See \cite[Lem. 8.2.1]{CG1}  and \cite[Lem. 9]{BCTO}. Also see \S \ref{ss:beta}.

\begin{proof}
The existence and regularity of $w,w^\sharp$ are discussed below \eqref{def:solM}.

The assumptions $\rD u(0)=\vec0$ and $u\big|_{\pa\cScmd{u,L}}=0$ apparently implies $\textstyle\inf_{\cScmd{u,L}}u=u(0)=-L$. 

Since $w$ is strictly convex, we can suppose $w(z )=\textstyle\inf_{\cScmd{u,L}}$ for some $z \in\cScmd{u,L}$ so that
\be\label{w:z:min}
\rD w(z )=\vec 0.
\ee
By \eqref{comp12} and \eqref{vw:mm:comp} of Corollary \ref{cor:est:com}, we find
  \be\label{vw:lem}
\big|u(0)- w(z ) \big|\le\textstyle\sup_{\cScmd{u,L}}\big| u - w \big| \le  \tfrac12n^2\delta,
\ee
which proves \eqref{vw:sn}.
  Then, by \eqref{est:wm} with $r_1=1$ and the assumed $\delta\le \tfrac15n^{-2}$, we show
  \begin{align}\label{est:vmin}
 -L=u(0)\le w(z ) +\tfrac12n^2 \delta\le- \tfrac25 .
\end{align}
 Also, \eqref{vw:sn} implies $\textstyle \sup_{\cScmd{u,L/2}} w(x)\le \sup_{\cScmd{u,L/2}}u(x)+\frac12n^2 \delta= -\frac12 L+\frac12n^2 \delta$. Then,  in view of \eqref{est:vmin} and the assumed $\delta\le \tfrac15n^{-2}$, we have $\sup_{\cScmd{u,L/2}} w(x)\le-\tfrac1{10}$, which allows us to apply   Lemma \ref{C4Bound} to interior estimates \eqref{Pog0:3i} and that, for a positive constant $C_3$ that only depends on $n$,
 \be\label{Pog3:a}
 |\rD^3 w|_{\cC^{0,1/2} } \leq C_3\qqu{in}\cScmd{u,L/2}.
 \ee 
 By interpolation of these two estimates,
 we prove \eqref{Pog0:3ii}.

 Next, by  \eqref{vw:lem}, we have $u(z)-u(0)=u(z) - w(z)+w(z)- u(0 )\le n^2\delta\le  \tfrac15$. Then, by \eqref{est:vmin}, we have  $u(z)-u(0)\le L/2$ and in other words, $z$ is in the closure of ${{\cScmd{u,L/2}}}$. Thus, by Taylor expansion and \eqref{w:z:min}, 
\[
w(0) - w(z ) = \tfrac{1}{2}  z ^t \mathrm{D}^2w\big|_{\xi_1}z \qquad \text{for some} \quad \xi_1 \in  \cScmd{u,L/2},
\]
 where the left-hand side equals $w(0) - u(0)+u(0)- w(z )\le  n^2\delta$ due to \eqref{vw:lem}. Thus, by \eqref{Pog0:3i}, 
 \be\label{x0x1}
|z |^2 \leq 2 c_2^{-1}n^2\delta.
 \ee
Combining this with \eqref{Pog0:3i} and \eqref{w:z:min}, we prove \eqref{gradest}.
  
Next, by \eqref{est:vmin}, the assumption $\hc\le\tfrac15$ ensures \(\cScmd{u,h}\subset\cScmd{u,L/2}\). Also, by \eqref{vw:lem}, we have 
\[
\big|w(x) -h - w(z ) \big| =\big|w(x) - u(x) + u(0)-w(z )\big|\le n^2 \delta,\qquad\forall\,x\in\pa\cScmd{u,h}.
\] 
Perform Taylor expansion on $ w(x) - w(z )$ with a 2nd derivative remainder that is bounded using \eqref{Pog0:3i}, noting also \eqref{w:z:min},   to find $\tfrac12 { c_2}|x-z |^2\le h+ n^2\delta$. Combining this with  \eqref{x0x1}  gives
\be\label{T:exp:2}
\left.\begin{aligned}
&|x| \le \sqrt{2}\, c_2^{-\frac12}\big(\sqrt{h+n^2 \delta}+\sqrt{n^2 \delta}\big)\le  C_7\,h^{\frac12}\,,\;\; \qquad\forall\, x\in\pa\cScmd{u,h},\\
&\qquad\qquad\text{where}\quad C_7:=\sqrt{2}\, c_2^{-\frac12}\big(\sqrt{1+n^2 C_4}+\sqrt{n^2 C_4}\big)\qquad\text{(due to \eqref{cond:scaling2})}.
\end{aligned}\quad\right\}
\ee

Yet again by \eqref{vw:lem}, 
\[
\big|w(x)-h - w(0) \big| =\big|w(x) - u(x) + u(0)-w(0)\big|\le n^2 \delta,\qquad\forall\,x\in\pa\cScmd{u,h}.
\] 
Perform Taylor expansion on $ w(x) - w(0)$ with a 3rd derivative remainder that is bounded using \eqref{Pog3:a} to find
\[
\Big|\frac{h}{2}|Tx|^2-h\Big|=\Big|\frac12 x^t\,\mathrm{D}^2w(0)\,x-h\Big|\le  n^2\delta+ \big|\rD w(0)\cdot x\big| +\frac16C_3|x|^3, \qquad\forall\, x\in\pa\cScmd{u,h}.
\]
By \eqref{gradest} and \eqref{T:exp:2}, the right-hand side is bounded by 
\be\label{C3:effect}
\Big( n^2\delta/h+nC_2C_7\big({2  c_2^{-1}\delta/h}\big)^{\frac12}+\tfrac16C_3C_7^3 h^\frac12\Big)\cdot h,
\ee
which  proves the $\cB_{r_-}  \subset  \,T\,\cScmd{u,h}  \subset \cB_{r_+}$ part of \eqref{n:d:1} with $r_\pm$ satisfying \eqref{n:d:2}. By choosing a suitable constant $ \hc$ as the upper bound for $h$ and a suitable constant $C_4$ as the upper bound for $\delta/h$, we show the first and last inclusions of \eqref{n:d:1}.
  
 For the final part of the proof, first, since $\det^{\frac1n}T=1/\sqrt h$ due to $\det\rD^2w=1$, combining \eqref{vw:lem} and transform estimates \eqref{T:Dk:al:rs} proves
\be\label{uw0Linfty}
 \big|\cF_T(u-w) \big| \leq \tfrac12n^2\delta/h \qqu{in}T\cScmd{u,h}.
\ee
By affine invariance \eqref{T:D2}, we have 
\[
\det\rD^2\big(\cF_T\,u\big)=g\circ T^{-1}\qqua{and}\det\rD^2\big(\cF_T\,w^\sharp\big)=1\qqu{ in}T\cScmd{u,h}.
\]
 We then validate assumptions \eqref{v:cond}, \eqref{cond:scaling2} on $\cF_T\,u$ and $g\circ T^{-1}$ in the domain $T\cScmd{u,h}$. Since $T\cScmd{u,h}$ is quasi-normalised due to \eqref{n:d:1}, we apply  \eqref{vw:sn} on $\cF_T(u+L-h)$ and  $\cF_T\,w^\sharp$  to find 
\be\label{uw1Linfty}
\big| \cF_T(u -w^\sharp)-const\big|\leq   \tfrac12n^2\delta\qqu{in}T\cScmd{u,h}.
 \ee
Also, considering $\cF_Tw^\sharp$ in domain $T\cScmd{u,h}$, we apply  \eqref{Pog0:3i}, \eqref{Pog3:a} to obtain $C^{1,\alpha_0}$ estimates of $\rD^2(\cF_Tw^\sharp)$ in  $T\cScmd{u,h/2}$. Combining the original form of \eqref{Pog0:3i}, \eqref{Pog3:a} on $w$ with transform estimate \eqref{T:Dk:al:rs} also gives the corresponding $C^{1,\alpha_0}$ estimates of $\rD^2\cF_Tw$ in $T\cScmd{u,L/2}\supset T\cScmd{u,h/2}$. Noting  that  only the $\ell=2,3$ cases of \eqref{T:Dk:al:rs} was used and noting  
 $\det^{\frac1n}T=1/\sqrt h$ again, we find these (upper) bounds depend on non-negative powers of $h$ and hence can be relaxed to be independent of $h$. Finally, since the convexity of $u$ implies $T\cScmd{u,h/2}\supset\frac12T\cScmd{u,h}$, we use \eqref{n:d:1} to have $ T\cScmd{u,h/2}\supset\cB_{1/2}$. Therefore, we
apply Lemma \ref{lem:ol} to $\cF_T\,w,\cF_T\,w^\sharp$ in the domain $\cB_{1/2}$ together with  \eqref{uw0Linfty} and \eqref{uw1Linfty}, noting that  $\rD^2(\cF_{T}\,w)(0)=I$ by the chain rule, to prove \eqref{est:D2w1w0}, \eqref{est:D3w1w0}. 
  \end{proof}

\section{\bf Iterative Steps} 
\label{sec:iterativeproof}

By iterating what was done in Lemma \ref{Stepk} going from a section to a smaller section about the same point, we construct a sequence of shrinking sections of the MAE solution and obtain bounds on the eccentricity of the affine transforms that quasi-normalise them. 


 \medskip

\noindent{\bf Assumptions}. Suppose  $\Vo$ is strictly convex and $\cC^2$ in a closed section  $\overline{\cScmd{\Vo,h_0}}$ for  $h_0>0$, and
\be\label{Om0:prop}
\cB_1(\xi) \subset \cScmd{\Vo,h_0}\subset \cB_n(\xi)\qqu{ for some }\xi\in\bR^n.
\ee
Also suppose 
 \be\label{wat}
\detn\rD^2 \Vo=\Fo\qu{in} \cScmd{\Vo,h_0}\qqua{with}\rD \Vo(0)=\vec0\qua{and} \Fo(0)=1.
\ee 
For constants $\hc$, $C_4$ and $C_5$ from  Lemma \ref{Stepk}, define
\begin{align}
\label{def:Cmc}\Cmc&:=\min\left\{C_4\hc\,,\,\hc/(3C_5)\,,\,\tfrac15n^{-2}\right\},\\
\delta_0&:=\textstyle\sup_{ \cScmd{\Vo,h_0}} |\Fo-\Fo(0)|,\nonumber
\end{align}
and suppose
\be\label{de0:Cmc}
\delta_0\le\Cmc.
\ee

\noindent{\bf Definitions}. For constant $ \hc$ from  Lemma \ref{Stepk}, define
\begin{align*}
w_{0}:=\solM{\cScmd{\Vo,h_0}}\qqua{ and }w_{k}:={\textsf{Sol}}\big(\cScmd{\Vo,\hc^k}\big)\qquad\;\;\forall\,k\ge1.
\end{align*}
Then, iteratively define   transforms $\pp_0$, $T_k$ and $\pp_{k+1}$ for all $k=0,1,2,\ldots$ as follows:
\begin{align}
&\pp_{0}:=I\,,\nonumber\\
&T_k:=\hc^{-{1\over2}}\sqrt{\mathrm{D}^2(\cF_{\pp_k}w_k)(0)}\,,\label{def:Tk}\\[1mm]
&\pp_{k+1}:=T_k\,\pp_{k}\,.\nonumber
\end{align}
Apparently,  for all $k\ge0$,
\be\label{det:T:P}
\detn T_k=\hc^{-\frac12}\qqua{and}\detn \pp_{k+1}=\hc^{-\frac{k+1}2}.
\ee
Finally,  define
\be\label{dk:def}
\delta_{k}:=\textstyle\sup_{ \cScmd{\Vo,\hc^k}}\big|\Fo-\Fo(0)\big| =\textstyle \sup_{\, \pp_k\cScmd{\Vo,\hc^k}}\big|\Fo\circ\pp_k^{-1}-\Fo(0)\big| ,\qquad\forall\,k\ge1,
\ee
so that, by \eqref{de0:Cmc} and $\cScmd{\Vo,\hc^k}\subset \cScmd{\Vo,h_0/2}$ from Lemma \ref{Stepk}, we have
\be\label{dk:d0:bound}
\delta_k\le\Cmc,\qquad\forall\,k\ge0.
\ee

\begin{lemma}
\label{lem:iter}
Under the set-up so far in \S\,\ref{sec:iterativeproof} and  the same constants as in  Lemma \ref{Stepk},  we have
 \begin{align} 
 \label{Om:Bpm}
 \cB_1 \subset\pp_{k}\,\cScmd{\Vo,\hc^{k}}&
\subset  \cB_2\,,\qquad\forall\,k\ge1,
\end{align}
and 
 \be\label{pk:ratio}
\mat{T_k^{-1}}\le\tfrac65\sqrt{\hc}<1\,,\;\;\quad\qquad\forall\,k\ge1.
 \ee 
 We also have what will be referred to as ``compound estimates''
 \be\label{c:e} 
 \left.
 \begin{aligned}
&\mt{\br{\pp}_{k}}^2 \le C_2 M_{k}\qu{ \; and \;  } \mt{\br{\pp}_{k}^{-1}}^2 \le  c_2^{-1} M_{k}\\[1mm]
& \qu{for}\;\;
  M_{k}:=\exp\Big(1.12\,{C_5\over \hc}\,\sum_{j=0}^{k-2}\delta_{j}\Big),
\end{aligned}\;\;
\right\}\qquad\forall\,k\ge1,
 \ee
where the sum is understood as $0$ if $k=1$, making $M_1=1$.

Finally,
there exists a constant $C_7$ that only depends on $n$ so that, for all $k\ge0$,
\begin{align}
\label{est:Tww1:int}
\big|\rD^2(w_{k+1}-w_k)(0)\big|&\le C_2(M_{k+2}-M_{k+1}) ,\\
\label{est3:Tww1:int}\Big|\rD^3\big(\cF_{\pp_{k+1}}(w_{k+1}-w_k)\big)\Big|&\le 
 C_7\,\delta_k/\hc\qqu{in}\cB_{1/4}.
\end{align}
 \end{lemma}

\begin{proof}
We prove \eqref{Om:Bpm} by induction. 
Fixing any $k\ge0$, suppose we have established the inductive hypothesis -- note when $k=0$, the inductive hypothesis is just \eqref{Om0:prop}. 
Invoke Lemma \ref{Stepk} where  the role of $\cScmd{u,L}$ is played by $\cScmd{\Vo,h_0}$ when $k=0$ and played by $\pp_k\,\cScmd{\Vo,\hc^k}$ when $k\ge1$, and 
\be\label{Stepk:roles}\begin{array}{rcccl}
\text{the roles of \; }&h\,,&\delta\,,&g\,,&u\\[1mm]
\text{are played by \; }&
\hc\,,  & {\delta_k}\,,  &  {\Fo\circ\pp_k^{-1}}\,, &  {\cF_{\pp_k}\Vo-const}\,.  
\end{array}
\ee
 Then, the definitions of $w_k,T_k$ means that 
\be\label{Stepk:roles:2}
\begin{array}{rcl}
\text{the roles of \; }&w\,,&T\qquad\\[1mm]
\text{are played by \; }& \cF_{\pp_k}w_k\, ,& T_k \,.\qquad
\end{array}
\ee   

 All assumptions in Lemma \ref{Stepk} are verified as follows.
  The quasi-normalisation assumption is satisfied due to the inductive hypothesis whereas assumption \eqref{v:cond} is satisfied due to \eqref{wat} and the roles of $u,g$ as assigned above. The definition of $\hc$ makes \eqref{cond:scaling1} satisfied whereas \eqref{dk:d0:bound} and \eqref{def:Cmc} make \eqref{cond:scaling2} satisfied.  
  With all assumptions verified and noting 
  \be\label{useful}
\text{the role of \; }\cS(u,h)\text{ \; is played by \; }\cS\big(\cF_{\pp_k}\Vo,\hc\big)=\pp_{k}\,\cS(\Vo,\hc^{k+1}) 
\ee
(the equality is due to property \eqref{T:sn} and determinant \eqref{det:T:P}), we invoke inclusions \eqref{n:d:1} on the above set to prove \eqref{Om:Bpm} for $\pp_{k+1}\cScmd{\Vo,\hc^{k+1}}$. The {\it{inductive step is complete}}.  

To prove the rest of the Lemma, let us use the same notations as within the above inductive step. 
Then, \eqref{useful} with the definition of $w_{k+1}$  implies
$\cF_{\pp_k}w_{k+1}$ plays the role of $\solM{\cS(u,h)}$, namely, $w^\sharp$ in  Lemma \ref{Stepk}, hence $\cF_{\pp_{k+1}}(w_{k+1}-w_k)$ plays the role of $\cF_T(w^\sharp-w)$. Then, we apply \eqref{est:D3w1w0}  to show  estimate \eqref{est3:Tww1:int} and apply  \eqref{est:D2w1w0} to show
\be\label{est:w1:I}
\Big|\rD^2(\cF_{\pp_{k+1}}w_{k+1}-\cF_{\pp_{k+1}}w_{k})(0)\Big|=\Big|\rD^2(\cF_{\pp_{k+1}}w_{k+1})(0)-I\Big|\le C_5\,\delta_k/\hc.
\ee   

Recalling definition \eqref{def:Tk} but for $T_{k+1}$ and using a straightforward eigenvalue analysis on \eqref{est:w1:I}, we have 
 $
 \mat{\br{T}_{k+1}-I}\le  \frac12 C_5\,\delta_k/\hc
 $
 which is further bounded by  $\frac16$ due to  \eqref{dk:d0:bound} and \eqref{def:Cmc}. Thus,  
\begin{alignat}{2}\label{Tk10:mat}
 \mt{\br{T}_{k+1}} &\le 1+ \tfrac12 C_5\,\delta_k/\hc \le \tfrac7{6} \qqua{and} \mt{\br{T}_{k+1}^{-1}} &\le \big(1- \tfrac12 C_5\,\delta_k/\hc\big)^{-1}\le \tfrac6{5}\,.
\end{alignat}
Combining it with $\hc\le\frac15$ from  Lemma \ref{Stepk}, we  show    \eqref{pk:ratio}. Also, by Taylor expansion, we have
 \[\begin{aligned}
 \ln\!\big(\mt{\br T_{k+1}}^2\big)&\le 2\ln(1+ \tfrac12 C_5\,\delta_k/\hc) <C_5\,\delta_k/\hc, \\
 \ln\!\big(\mt{\br T_{k+1}}^2\big)&\le 2\ln(1- \tfrac12 C_5\,\delta_k/\hc)^{-1} <1.12\,C_5\,\delta_k/\hc.
 \end{aligned} 
\] 
 Since $\mt{\br{T}_0}^2\le{C_2}$ and $\mt{\br{T}_0^{-1}}^2\le  c_2^{-1}$ due to \eqref{Pog0:3i}, this proves \eqref{c:e}.
  
To prove \eqref{est:Tww1:int}, recall estimates  \eqref{est:w1:I}  with its right-hand side further bounded by 
\[
C_5\,\delta_k/\hc\le \ln\big(\tfrac{M_{k+2}}{M_{k+1}}\big)\le\tfrac{M_{k+2}}{M_{k+1}}-1.
\]
Then, substituting it into  \eqref{est:w1:I}, applying transform estimates \eqref{T:Dk:al:rs} and  combining the result with  \eqref{c:e} applied on $\pp_{k+1}$ proves \eqref{est:Tww1:int}.
\end{proof}
 
  \subsection{Consequences of Dini continuity}\label{sec:C2}
 $\;$ 

When $f$ is Dini continuous, estimate \eqref{est:Tww1:int} allows us to approximate $\rD^2\Vo(0)$ using sequence $\{\rD^2 w_k(0)\}$, leading to an explicit upper bound on $\rD^2\Vo(0)$. For conciseness however, we devise the next lemma so that it requires an assumption, namely \eqref{Dini:alt:as}, that is not exactly Dini continuity but rather a consequence of it. This latter link will be established in \S \ref{sec:mainproof}.

\begin{lemma}\label{lem:Dini}
Under the set-up of Iteration Lemma \ref{lem:iter}, if further
 \be\label{Dini:alt:as}
 M_\infty:=\displaystyle\lim_{k\to\infty}M_k<\infty,
 \ee
 then
 \be
\label{Dini:cons:1}\big|{\rD^2 \Vo(0)}-\rD^2w_k(0)\big|<C_2\,(M_\infty-M_{k+1}),\qquad\forall\,k\ge0,
\ee
and
\be\label{Dini:cons:2}
\big|\rD^2 \Vo(0)\big|\le C_2\,M_\infty .
\ee
\end{lemma}
\begin{proof}

We will use without reference the $\ell=2$, $\alpha=0$ version of \eqref{T:Dk:al:rs}. 

Since  $\{M_k\}$ is an increasing sequence, estimate \eqref{est:Tww1:int} and assumption \eqref{Dini:alt:as} imply 
  $\big\{\rD^2 w_k(0)\big\}$ is a Cauchy sequence and it remains to show the limit  of this Cauchy sequence is  $\rD^2 \Vo(0)$.

In fact, since  $\diam(\cScmd{\Vo,\hc^{k}})\le 2\mat{\pp_{k}^{-1}}$ due to \eqref{Om:Bpm}, we use \eqref{pk:ratio} to deduce  
\be\label{diam:lim}
\diam(\cScmd{\Vo,\hc^{k}})=o(1)\qqu{as}k\to\infty.
\ee Then, by the $\cC^2$ regularity of $\Vo$,  the oscillation of $\rD^2\Vo$ in the domain $\cScmd{\Vo,\hc^{k}}$ vanishes as $k\to\infty$. Thus, in view of \eqref{Dini:alt:as} and \eqref{c:e}, we have  
 \be\label{D2u:o1}
\textstyle\sup_{\pp_{k}\,\cScmd{\Vo,\hc^{k}}} \big|\rD^2 (\cF_{\pp_{k}}\Vo)(x)-\rD^2 (\cF_{\pp_{k}}\Vo)(0)\big|=o(1)\qqu{as}k\to\infty.
 \ee 
 Second, by \eqref{dk:def} and \eqref{diam:lim}, we have $\delta_{k} \to0$ which means, for any integer $\qq\gg1$ we can find large enough integer $K_i$ so that
 \be\label{choose:ell}
  \delta_{K_i}\le C_4\, \hc^{3\qq} \,.
 \ee
Using $K=K_i$ for brevity, we invoke  Lemma \ref{Stepk} with 
\[\begin{array}{rccccl}
\text{the roles of \; }&\cScmd{u,L}\,,&h\,,&\delta\,,&u\,,&w\,,\\[1mm]
\text{played by \; }&\pp_K\cScmd{\Vo,\hc^K}\,,&
\hc^{2\qq+1}\,,  & {\delta_{K}}\,,  &  {\cF_{\pp_K}\Vo-const}\,,&\cF_{\pp_K}w_{K}\,. 
\end{array}
\] 
Most assumptions of Lemma \ref{Stepk} are validated in the same way as in the proof of Lemma \ref{lem:iter} except that it is \eqref{choose:ell} that implies $\delta=\delta_K$ satisfies assumption \eqref{cond:scaling2}. The role of  $T$ of Lemma \ref{Stepk} is played by 
$\big(\hc^{2i+1}\big)^{-\frac12} \sqrt{\mathrm{D}^2(\pp_Kw_{K})(0)}=\hc^{-i}\,{T}_{K}$ (recall definition \eqref{def:Tk} for $T_K$).  
Then,  recalling \eqref{T:sn} and \eqref{det:T:P}, 
we have that the role of $T\cS(u,h)$ is played by   
\[
\hc^{-i} \,{T}_{K}\cS(\cF_{\pp_K}\Vo\,,\,\hc^{2\qq+1})=\hc^{-i} \,\cS(\cF_{\pp_{K+1}} \Vo, \hc^{2i}) .
\]
Then, by inclusions \eqref{n:d:1} and estimates \eqref{n:d:2}, we obtain
\[
\begin{aligned}   
\hc^i\,\cB_{r_{-}} \subset \cS(\cF_{\pp_{K+1}}\Vo,  \hc^{2i})&\subset \hc^i\,\cB_{r_{+}}
\end{aligned}
\]
with $\big|r_\pm^2-{2}\big|=O(\hc^{{\qq}})$ due to \eqref{choose:ell} and the role of $h$ assigned above. Therefore, for any unit vector $\ve\in\bR^n$, there exists a scalar $r'$ so that  $(r')^2=2+O(\hc^{{\qq}}) $ and 
   $r'\hc^i\ve$ is on the boundary of the above middle set. Then, by  Taylor expansion and \eqref{D2u:o1}, recalling $K=K_i$, we find
\begin{equation*}\begin{aligned}
\hc^{2i}
&=\tfrac12\big(2+O(\hc^{{\qq}})\big)\,(\hc^{i}\ve)^{t}\big(\rD^2 (\cF_{\pp_{K_i+1}}\Vo)(0)+o(1)\big)(\hc^{i}\ve) \qqu{as}i\to\infty.\end{aligned}
\end{equation*}
Since this holds for any unit vector $\ve$, we have
\(
\rD^2  \big(\cF_{\pp_{K_i+1}}\Vo\big)(0)-I=o(1)
\) as $i\to\infty$.
 Then, in view of \eqref{est:w1:I} with $\delta_k\to0$ and the uniform bound \eqref{Dini:alt:as}, we prove that  the Cauchy sequence $\big\{\rD^2  w_k(0)\big\}$ indeed tends to $\rD^2 \Vo(0)$, hence completing the proof of \eqref{Dini:cons:1}.  Then, combining the $k=0$ case of \eqref{Dini:cons:1} and the upper bound in \eqref{Pog0:3i} applied to $w_0$ proves the  $\cC^2$ estimate \eqref{Dini:cons:2}.
\end{proof}

 \section{\bf Estimates on the First Section}\label{sec:first:step}
 Since the Iteration Lemma \ref{lem:iter} requires in \eqref{de0:Cmc} small oscillation of the right-hand side function over the first section in the sequence, we now focus on the existence of such first section in the interior of the domain. We also prove estimates on the affine transform that quasi-normalises the first section, which will be needed in piecing together the proof of the Main Theorem (see \eqref{est:t:tau}). Caffarelli's strict convexity result (\cite{LC4}) lays the foundation of the following proof. 

 
\begin{lemma}\label{lem:strict}
Suppose $\Om \subset \bR^n$ is a bounded, open and convex set satisfying $\cB_1 \subset \Om \subset \cB_n$ and let $v \in \cC^2(\Om)\cap \cC^0(\overline{\Om})$ be a convex solution to   \eqref{MA1}, \eqref{MA2} with $1\le f\le \fm  $ in $\Om$ for constant $\fm$. Then, for any fixed $0< \rhob <\frac12\diam(\Om)$, there exists constants $H=H(n,\fm,\rho)>0$ and
positive $\theta=\theta(n,\fm)<1$ 
 so that for every integer $j\ge0$, 
  \begin{equation}\label{first:diam}
\cS(v,2^{-j}H,\xs)\subset \cB_{\theta^j\rho}(\xs)\subset\Om,\qquad\forall\,\xs\in\Omrho.
 \end{equation}
 \end{lemma}

 The proof is inspired by techniques from \cite[Lem. 5.4.1]{CG1}, \cite[Lem. 4.12]{AFB2}.

\begin{proof}
  
We first prove the $j=0$ case of \eqref{first:diam}. Suppose not.
 Then there exist a sequence of convex domains $\{\Om^i\}$ for large integers $i$. Each domain is sandwiched by $\cB_1$ and  $\cB_n$, and  there exists a $\cC^2$ convex function $v_i$ defined on $\Om^i$ that satisfies $1\le \detn \mathrm D^2 v_i\le \fm$. Also there exists a point $\xs_i\in(\Om^i)_\rho$ and a point $z_i\in\pa\cB_{\rhob}(\xs_i)$ so that
\begin{equation}\label{supposenot}
0\le v_i(y)-v_i(\xs_i)-\rD v_i(\xs_i)\cdot(y-\xs_i)<i^{-1},\qu{for}y=t\,\xs_i+(1-t)z_i\,,\quad\forall\,t\in(0,1).
\end{equation}
The  Blaschke selection theorem ensures the compactness of $\{\Om^i\}$ in Hausdorff metric, say $\Om^i\to\Om^\infty$ which is a convex set also sandwiched by $\cB_1$ and  $\cB_n$. Combining the global H\"older estimate of Corollary \ref{cor:sep} and the Arz\'ela-Ascoli lemma gives the compactness of $\{v_i\}$ on $\overline{\Om^k}\cap\overline{\Om^{k+1}}\cap\cdots$ for any single integer $k$. Through a standard diagonalisation technique, we can find a subsequence, also named $\{v_i\}$ for simplicity, that converges uniformly to a convex function $v_\infty$ on {\it all} compact subsets of $\Om^\infty$. Corollary \ref{cor:sep} also implies $v_\infty$ is continuous at and vanishes on $\pa\Om^\infty$. Moreover, standard stability result (e.g., \cite[Lem. 1.2.3]{CG1}, \cite[Prop. 2.6]{AFB2}) ensures the weak-* convergence of $\det \mathrm D^2v_i(x)\mathrm dx=\cM v_i$  to the the Monge-Amp\`ere measure $\mathcal M v_\infty$ in $\Om^\infty$. Thus $\mathrm dx\le {\mathcal M}v_\infty\le  \fm^n \rd x$ in $\Om^\infty$. By \eqref{mm:strict}, $v_\infty$ is strictly convex in the interior of $\Om^\infty$.

On the other hand, $\{\xs_i\}$, $\{z_i\}$ are apparently compact. Also, by the Lipschitz bound \eqref{Lipbound}, we have that $\{\nabla v_i( \xs_i)\}$ are compact with a subsequence converging to $p_\infty$. Then, taking the limit of \eqref{supposenot} shows exists a point $\xs_\infty\in(\Om^\infty)_\rho$ and a point $z_\infty\in\pa\cB_{\rhob}(\xs_\infty)$ so that
\[
 v_\infty(y)-v_\infty(\xs_\infty)-p_\infty\cdot(y-\xs_\infty)=0,\qu{for}y=t\,\xs_\infty+(1-t)z_\infty,\quad\forall\,t\in(0,1).
\]
 which contradicts the {\it strict} convexity of $v_\infty$. Therefore, we have proven the $j=0$ case of \eqref{first:diam}.

In proving the $j\ge1$ cases of \eqref{first:diam}, we first claim that for the same constant $\theta$ and any $h\in(0, H]$,
 \begin{equation}\label{sec:ratio}
 \begin{aligned}
\frac{|\xs-y_{\frac12}|}{|\xs-y_{1}|} \le  \theta,\qquad\forall& \,\text{co-linear points $\xs,y_{1},y_{\frac12}$ with }y_{1}\in\pa\cS(v,h,\xs)\\
&\text{ and } y_{\frac12}\in\pa\cS(v,\tfrac12 h,\xs)\text{ between }\xs,y_{1}.
\end{aligned}
\end{equation}
In fact, the ratio on the left-hand side is apparently invariant under invertible affine transforms. Affine invariance of MAE means that  the assumptions on $\xs,y_{1},y_{\frac12}$ in \eqref{sec:ratio} as well as the range $f\in[1,\fm]$ also remain unchanged up to a harmless redefinition of $h$. So, it suffices to prove \eqref{sec:ratio} in a normalised domain  $\Om'=\cS(v,h,\xs)$ with $v\big|_{\pa\Om'}=0$.   Then, by maximum principle Lemma \ref{lem:max},
\begin{align*}
|y_{\frac12}-y_{1}|\ge\dist(y_{\frac12},\pa\Om')&\ge \frac{h^n2^{-n}}{C_{\!\mathsf A}\diam(\Om')^{n-1}\int_{\Om'}f^n}.
\end{align*}
Next, the comparison principle Lemma \ref{lem:comp} implies $\frac12(|y|^2-n^2)\fm\le v(y)$ for any $y\in\Om'$, and thus noting $v(\xs)=-h$ gives $|\xs|^2\le n^2-2h/\fm$, which implies
\[
|\xs-y_{1}|\le n+(n^2-2h/\fm)^\frac12.
\]
By Lemma \ref{lem:comp}, we also have $\frac12(|y|^2-1)\ge v(y)$ for any $y\in\Om'$, and thus $h=-\min_{\Om'}v\ge \frac12$. 
Therefore, gathering all the estimates so far gives
 \begin{align*}
 \frac{|y_{\frac12}-y_{1}|}{|\xs-y_{1}|}&\ge \frac{2^{-2n}}{C_{\!\mathsf A}(2n)^{n-1}|\cB_n|\fm^n (n+(n^2-\fm^{-1})^\frac12)}\stackrel{def}=1-\theta.
 \end{align*} Hence, by co-linearity of $\xs,y_{1},y_{\frac12}$ and that $y_{\frac12}$ is between $\xs,y_{1}$, we have 
${|\xs-y_{\frac12}|}/{|\xs-y_{1}|}$ bounded from above by $\theta=\theta\big(n,\fm \big)$, 
whence proving \eqref{sec:ratio}.  

Since the argument below \eqref{sec:ratio} shows that it also applies to  the un-normalised section $\cS(v,2^{-j}H,\xs)$, we  prove \eqref{first:diam} by combinging  \eqref{sec:ratio} with the previously proven $\cS(v,h,\xs)\subset \cB_\rho(\xs)$. 
\end{proof}

\noindent{\bf Remark.} The above lemma is of interior estimates. If point $\xs$ were allowed to be arbitrarily close to the boundary, then in the first half of the proof, $\xs_\infty,z_\infty$ could possibly be situated on $\pa\Om^\infty$ and since there is no a priori information on the curvature of the boundary, the part of $\pa\Om^\infty$ between them could be possibly a straight segment, which would make the contradiction (to strict convexity) approach fail to work.

 \begin{lemma}\label{lem:sec:size}
Consider convex domain $\Om\subset\bR^n$ and  let $v \in \cC^2(\Om)\cap \cC^0(\overline\Om)$ be a convex solution to   $\detn\rD^2 v=f$ over $\Om$. 
 Consider a section  $\cS(v,h,\xs) \subset\Om$ and suppose 
 \(
 1\le f\le \fm 
 \)  in $\cS(v,h,\xs)$ for constant $\fm$.   Suppose an affine transform $T$  satisfies   
 \begin{equation}\label{qn:tr}
 \cB_{r_1}(\xi_1)\subset T\cS(v,h,\xs)\subset \cB_{r_2}(\xi_2)\,,
 \end{equation} for some $0<r_1\le r_2$ and $\xi_1,\xi_2\in\bR^n$. Then, 
 \begin{align}\label{est:det:A}
 \frac 12  r_1^{2} \le (\detns T)\, h \le \frac 12 r_2^2\fm\,,\\
\label{diam:up}   \frac{r_1}{r_2}\le \frac{(8h)^{1\over2}\mt{\br{T}^{-1}}}{ \diam(\cS(v,h,\xs) )} \le  \frac{r_2}{r_1}\fm^{\frac 1{2}}\,,\\[1mm]
\label{est:T1}  \mat{{T}^{-1}}\le \diam\big(\cS(v,h,\xs)\big)/{2r_1}\,.
  \end{align}  
 \end{lemma}
\begin{proof} Define 
 \begin{equation}\label{affine:example}
 \hat v(Tz):=\cF_T\big(v(z)-\rD v(\xs)\cdot(z-\xs)-v(\xs)-h\big)\qqu{for} z\in \cS(v,h,\xs),
 \end{equation}
which, by affine invariance of MAE, satisfies 
 \begin{equation}\label{affine:example:eq}
 \det \rD^2\hat v=f\circ T^{-1} \qqu{{in}} T\cS(v,h,\xs) \qqu{with} \hat v\big|_{\pa(T\,\cS(v,h,\xs))}=0.
\end{equation}
Thus, $\inf_{T\cS(v,h,\xs)}\hat v=-(\det T)^{2\over n} h $.
 By  following almost the same proof of the minimum estimate \eqref{est:wm} from Corollary \ref{cor:est:com}, we find upper and lower bounds on $\inf_{T\cS(v,h,\xs)}\hat v$,  whence proving \eqref{est:det:A}.
 
 Next, applying $T^{-1}$ on 
the three sets in \eqref{qn:tr} will stretch $\cB_{r_2}(\xi_2)$ along any direction by a factor of at most $\mat{{T}^{-1}}$ and  stretch $\cB_{r_1}(\xi_1)$ along a certain direction by a factor of exactly $\mat{{T}^{-1}}$. Therefore 
\[
2r_1\mat{{T}^{-1}}\le \diam\big(\cS(v,h,\xs)\big)\le 2 r_2\mat{{T}^{-1}},
\]
proving \eqref{est:T1}. 
Since $\mat{{T}^{-1}}=\mt{\br{T}^{-1}}(\det T)^{-\frac1n}$, combining the above with \eqref{est:det:A}  proves \eqref{diam:up}.
\end{proof}
 

  \section{\bf Proof of Theorem \ref{mainresult}}
 \label{sec:mainproof}

   Recall    $1\le  f \le \fm$ in $\Om$. For a given $0<\rho<\frac12\diam(\Om)$,  recall the positive $\theta=\theta(n,\fm)<1$ from Lemma \ref{lem:strict}. Recall $\qin$ from \eqref{choose:first} with $\Cmc$ given in \eqref{def:Cmc}.
Define integer
$J:=\left\lceil{\ln(\rho/\qin)}/{\ln\big(\theta^{-1}\big)}\right\rceil$ which is non-negative due to $\qin\le\rho$ and $\theta<1$. Then,  Lemma \ref{lem:strict} together with $\theta^{J}\le\qin/\rho$ implies for constant $H=H(n,\fm,\rho)$,
  \begin{align}
\label{incl:first}  \cS(v,2^{-J}\!H,\xs)\subset \cB_{\qin}(\xs)\subset\Om,\qquad\forall\,\xs\in\Omrho.
\end{align}
We {\bf fix} this point $\xs$ from now on and define
\[
\text{constant \; }b:= f (\xs)\ge1.
\]

Let $A$ be the matrix of the affine transform that normalises $\cS(v,2^{-J}\!H,\xs)$. Define transform $\tau$ 
\begin{align}\label{def:Tin}
\Tin z&:=A(z-\xs),
\end{align}
which makes $\Tin\, \xs=0$ and $\Tini 0=\xs$. By commutation property \eqref{T:sn}, there exists constant $h_0>0$ (the value of which is irrelevant) 
so that
\begin{align}\label{firstdomain:prop}
&\Tin\,\cS(v,2^{-J}\!H,\xs)=\cScmd{\Vo,h_0},\\
\label{vv}\text{for}\qquad &\Vo(\Tin z):=b^{-1}\cF_\Tin \big(v(z)-\rD v(\xs)\cdot(z-\xs)\big),
\end{align}
and section $\cScmd{\Vo,h_0}$ is quasi-normalised modulo a translation. Just to be clear, we have
\be\label{f:Phi}
\detn \Vo=\Fo\qu{in}\cScmd{\Vo,h_0}
\qqu{with}\Fo:= b^{-1}f \circ\Tini\qu{so that}\Fo(\tau z)= b^{-1}f(z).
\ee Then,  assumptions \eqref{Om0:prop}, \eqref{wat} are validated. 

Also, assumption \eqref{de0:Cmc} is validated due to \eqref{choose:first}, \eqref{incl:first} and $ b\ge1$. 
 
In estimates \eqref{est:det:A} and \eqref{diam:up} of Lemma \ref{lem:sec:size}, let the role of $h$ be played by $2^{-J}\!H(n,\fm,\rho)$ with  $2^{J-1}<2^{{\ln(\rho/\qin)}/{\ln(\theta^{-1})}}=(\rho/\qin)^{{\ln2}/{\ln(\theta^{-1})}}\le 2^J$, so that combining it with \eqref{sc:A:A1},  \eqref{incl:first} gives,
\be\label{est:t:tau}
\left.\begin{aligned}& \text{all upper/lower bounds of   $\mt{\bTin}$, $\mt{\bTini}$ and $\detn\tau$}\\
&\text{can be expressed as $(n,\fm)$-dependent positive powers of}\\
 &\qin^{\,-1} \text{ times $(n,\fm,\rho)$-dependent positive constants.}
\end{aligned}\qquad
\right\}\ee
 
\noindent{\bf Convention}. Let  $\Qin$ denote various values that satisfy dependence \eqref{est:t:tau}, which is consistent with its definition in  Theorem \ref{mainresult}. 

A useful estimate: since by \eqref{incl:first}, \eqref{firstdomain:prop} and $\cScmd{\Vo,\hc}\subset\cScmd{\Vo,h_0}$, we have $\Tini\cScmd{\Vo,\hc}\subset\cB_{\qin}(\xs)$, and  since  \eqref{Om:Bpm} shows $\pp_1\cScmd{\Vo,\hc}=\pp_1\Tin(\Tini\cScmd{\Vo,\hc})\supset\cB_1$, we apply  \eqref{est:T1}  to obtain
\be\label{Tini:p1:qin}
\mat{\Tini\pp_1^{-1}}\le  \qin.
\ee

 Now consider $\xs,\xs'\in\Omrho$ with   
 \[
 \ds :=  |\xs-\xs'|>0.
 \]
  The $ \Qin\,\ds \ge1$ case of  \eqref{mainestimate:1}  follows directly from $\cC^1$ estimate \eqref{Lipbound} and maximum principle Lemma \ref{lem:max}. The $ \Qin\,\ds \ge1 $ case of \eqref{mainestimate} follows directly from the $\cC^2$ bound \eqref{C2:est}. 
 So it suffices to consider the case of $\Qin\,\ds  < 1$ from here on. By \eqref{c:e} and \eqref{est:t:tau}, we can always choose a suitable $\Qin$ in the condition $ \Qin\,\ds  < 1$ so that \(\ds <\tfrac12\,\mat{\pp_{1}\Tin}^{-1}\).  Also, \eqref{pk:ratio} implies $\big\{\mat{\pp_{k}\Tin}\big\}_{k\ge1}$ is strictly increasing. Therefore, we  obtain that
 \be\label{cond:rk}
 \tfrac12\,\mat{\pp_{k+1}\Tin}^{-1}\le  \ds <\tfrac12\,\mat{\pp_{k}\Tin}^{-1}\qu{ for unique integer}k\ge1,
 \ee
 and {\bf fix} this $k\ge1$ from now on.
 
 Next, we establish some useful estimates in terms of $\ds$. Using
 \[
1\le \mat{\Tini\pp_{k}^{-1}}\cdot \mat{\pp_{k}\Tin}=\mt{\bTini\breve\pp_{k}^{-1}}\cdot \mt{\breve\pp_{k}\bTin}\le \Qin M_{k}
 \]
where the upper bound was due to \eqref{c:e} and \eqref{est:t:tau}, we can obtain an upper bound for $\mat{\Tini\pp_{k}^{-1}}$ by considering \eqref{cond:rk}  and $ \mat{\pp_{k+1}\Tin}\le \tfrac76\mat{\pp_{k}\Tin}/ \sqrt\hc$ (due to \eqref{Tk10:mat}). Apparently, we can also use \eqref{cond:rk} to obtain a lower bound. In short,
\be
\label{Rr:ratio}
2\ds<\mat{\Tini\pp_{k}^{-1}}\le \Qin\, M_{k+1}\ds\,.
\ee
Also, combining \eqref{c:e}, \eqref{est:t:tau}  and \eqref{cond:rk} shows
\be\label{def:r:imply}
\detns\big(\Tini\pp_{k}^{-1}\big)=\dfrac{ \mt{\breve\pp_{k+1}\bTin}^2}{\hc\,\mat{\pp_{k+1}\Tin}^2}\le\Qin\, M_{k+1}\ds^{\,2}\,.
\ee

We move on to $\delta_k$ and $M_k$. Since \eqref{Om:Bpm} implies $\Tini\cScmd{\Vo,\hc^{k}}
\subset  \cB_{2\mat{\Tini\pp_{k}^{-1}}}$, we combine this with definition \eqref{dk:def} of $\delta_k$ and 
definition \eqref{f:Phi} of $\Fo$, noting bound \eqref{dk:d0:bound}, to have
\be\label{dk:bound}
\delta_k\le \min\Big\{\mc_{ f |_{\cB(\qin,\xs)}}\big(2\mat{\Tini\pp_{k}^{-1}}\big),\,\Cmc\Big\},\qquad\forall\,k\ge1.
\ee
Then, $\frac{\ln M_{2}}{1.12\,{C_5/ \hc}}\le \Cmc$,  and for $k\ge2$, we use monotonicity of $\wt\mc_ f $ and \eqref{pk:ratio} to find
\[
\begin{aligned}
\frac{\ln M_{k+1}}{1.12\,{C_5/ \hc}}\le \Cmc+\sum_{j=1}^{k-1}\delta_k
&\le \Cmc+\frac{-1}{\ln(\frac65\sqrt{\hc})}\int_{2\mat{\Tini\pp_{k-1}^{-1}}}^{C \mat{\Tini\pp_1^{-1}}}\min\Big\{\mc_{ f |_{\cB(\qin,\xs)}}(q),\,\Cmc\Big\}\,{\rd q\over q}\\
&\le \Cmc+\frac{-1}{\ln(\frac65\sqrt{\hc})}\Big(\int_{\qin}^{C\qin}{\Cmc}\,{\rd q\over q}+\int_{\bar C\ds}^{\qin}{{\mc_{ f }(q)}}\,{\rd q\over q}\Big),
\end{aligned}
\]
where $\bar C\ds$ was due to \eqref{Rr:ratio} and \eqref{pk:ratio}, and $C\qin$ was due to \eqref{Tini:p1:qin}. Then, 
by  defining 
\be\label{def:beta}
\beta:=1.12\,C_5/ \big(-\hc\ln(\tfrac65\sqrt{\hc})\big),
\ee
and recalling definition of $\cE_{\ds}$ from Theorem \ref{mainresult}, we prove
\be\label{Mk:bd}
M_{k+1}\le  C\cE_{\ds}
\ee
for constant $C$ that only depends on $n$. 
The above calculation naturally implies
\be\label{Dini:M:inf}
 f  \text{ being Dini continuous in $\cB_{\qin}(\xs)$\; implies \; }M_\infty=\lim_{k\to\infty}M_k\le C\cE_{0}<\infty.
\ee
 Also, in view of \eqref{Rr:ratio}, \eqref{dk:bound} and \eqref{Mk:bd}, we have
 \be\label{dk:final}
 \delta_k\le \mc_ f (\Qin \cE_{\ds}\, \ds).
 \ee
 
In the final stage of the proof, we introduce $\uu $ to be the ``backward transform''  of  $ w_{k}$ from domain $\cScmd{\Vo,\hc^{k}}$ onto   $\Tini\cScmd{\Vo,\hc^{k}}$ while also ``adding back" the information of the gradient of $v$,
\be\label{def:uk}
\uu (z)\big|_{\Tini\cScmd{\Vo,\hc^{k}}}:=\cF_{\Tini}\big(w_{k}\big|_{\cScmd{\Vo,\hc^{k}}}\big) \,+\,b^{-1}\rD v(\xs)\cdot(z-\xs). 
\ee 
Then, in view of \eqref{vv}, we have
 \be\label{uvw:rel}
 v-b\,\uu=b\,\cF_{\Tin^{-1}}(\Vo-w_{k}).
 \ee
 Recall definition $w_{k}=\textsf{Sol}(\cScmd{\Vo,\hc^k})$. By the quasi-normalisation \eqref{Om:Bpm} and estimate \eqref{dk:d0:bound} on $\delta_k$, the assumptions for Lemma \ref{Stepk} are verified with the roles of $u,w$ played according to \eqref{Stepk:roles} and \eqref{Stepk:roles:2}, so that we can invoke \eqref{vw:sn}\,-\,\eqref{gradest} followed by transforming the domain from $\pp_k\,\cScmd{\Vo,\hc^k}$ back to $\Tini\cScmd{\Vo,\hc^k}\supset\Tini\pp_{k}^{-1}\cB_1$ (which should be combined with $\cScmd{\Vo, \hc^k/2}\supset\frac12\cScmd{\Vo, \hc^k}$ by convexity of $v$) and applying transform estimate \eqref{T:Dk:al} or \eqref{T:Dk:al:rs} together with using \eqref{cond:rk} to relax factors of $\mat{\pp_{k}\Tin}$ and using \eqref{def:r:imply}, \eqref{Mk:bd} to relax factors of $\detns(\Tini\pp_{k}^{-1})$, to obtain
\begin{align}\label{back2Thm:0}
| v -b\,\uu-const|& \le  \Qin \,\cE_{\ds}\,\ds^{\,2}\delta_k \qqu{in}\cB_{2\ds}(\xs),\\
\label{back2Thm:2} (\Qin\, \cE_{\ds})^{-1}I\le \rD^2\uu&\le \Qin\, \cE_{\ds} \,I \;\,\quad\qqu{in}\cB_{\ds}  (\xs),\\
\label{back2Thm:25}|\rD^2\uu|_{\cC^{1/2}}&\le \Qin \,\cE_{\ds}\,\ds^{\,-\frac12}\,\qqu{in}\cB_{\ds} (\xs),\\
\label{back2Thm:1}|\rD v(\xs)-b\,\rD \uu(\xs)|&\le \Qin \,\cE_{\ds}\,\ds\,\delta_k^{\frac12}\,,
\end{align}
where in the last estimate we also used $\rD \Vo(0)=\vec 0$.


Define $v'$ in the same fashion as \eqref{firstdomain:prop} but in a neighbourhood of $\xs'$, and define $\uup$ in the same fashion as \eqref{def:uk} but in terms of $v'$, $\xs'$ and constant $b'= f (\xs')$. 
Let
\[
\xm:=\tfrac12 (\xs+\xs').
\] 

In the rest of \S \ref{sec:mainproof}, we will use  transform estimates \eqref{T:Dk:al:rs} without referencing.

\subsection{Second derivative estimate}\label{D2:proof}Suppose $ f $ is Dini continuous in $\cB_{\qin}(\xs)$ and $\cB_{\qin}(\xs')$.

Split the difference in the values of $\rD^2 v$ at $\xs$, $\xs'$
\begin{align*}
\big|\rD^2 v(\xs)-\rD^2 v(\xs')\big|
&\le \Big(\big|\rD^2 v(\xs)-b\,\rD^2 \uu(\xs) \big| +b\big|\rD^2 \uu(\xs)-\rD^2 \uu(\xm) \big|\Big)\\
&\quad\;\;
+\big(\text{counterparts in terms of } \xs', \uup, b'\;\big)\\
&\quad\;\;+\Big|b\,\rD^2\uu(\xm)-b'\,\rD^2\uup(\xm)\Big|,
\end{align*}
so that, by symmetry and without loss of generality, 
\be\label{bridge:est}
\begin{aligned}
 \big|\rD^2 v(\xs)-\rD^2 v(\xs')\big|&\le 2\Big|\rD^2 v(\xs)-b\,\rD^2 \uu(\xs) \Big| 
 + \fm\, \ds^{ }\, 
 \big|\rD^3\uu\big|_{\cC^0\left(\cB( \frac12\ds  ,\,\xs)\right)}\\
& \quad +\Big|b\,\rD^2\uu(\xm)-b'\,\rD^2\uup(\xm)\Big|.
\end{aligned}
\ee

We start with estimating the last term.  By definition of $\xm$, we have  $\cB_{\ds/2  } (\xm)\subset \cB_{\ds } (\xs)$ and  $\cB_{ \ds /2 } (\xm)\subset \cB_{\ds } (\xs')$.
 By symmetry and without loss of generality, we can have $\uup$ satisfy  \eqref{back2Thm:0}\,--\,\eqref{back2Thm:25} in the same domain $\cB_{\ds /2} (\xm)$ except that  $b$ should be replaced by $b'$. Then, by \eqref{back2Thm:0},
\be\label{back2Thm:0:ww}
| b\,\uu-b' \uup -const|  \le  \Qin \,\cE_{\ds}\,\ds^{\,2}\delta_k\qqu{in}\cB_{\ds /2}(\xm)\,.
\ee
By $b= f (\xs)$, $b'= f (\xs')$, we have $|b-b'|\le \mc_ f (\ds)$.
 Therefore, by invoking Lemma \ref{lem:ol} with the role of $\delta$ played by $ \Qin \,\cE_{\ds}\,\ds^{\,2}\delta_k$ and $r,r_0$ played by $\frac12\ds,\frac14\ds$, and combining the result with \eqref{bw:ww}, \eqref{dk:final}, \eqref{back2Thm:2}, \eqref{back2Thm:25}, we show
 \be\label{o:est:2}
 \big|b\,\rD^2 \uu-b'\,\rD^2\uup\big|\le  \Qin \, \cE_{\ds}^{\beta_1}\mc_ f (\Qin \cE_{\ds}\, \ds)\qqu{in}\cB_{\ds /4} (\xm).
\ee 

To estimate the second last term of \eqref{bridge:est},  we combine \eqref{est3:Tww1:int}, \eqref{est:t:tau} and \eqref{Mk:bd} to find
\[\begin{aligned}
\Big|\rD^3\big(\cF_{\Tini}(w_{j+1}-w_{j})\big)\Big|&\le  \Qin\, \cE_{\ds}^{1.5}\,\delta_{j}\big/\detn(\Tini\pp_{j}^{-1})\le  \Qin\, \cE_{\ds}^{2}\,\delta_{j}\big/\mat{\Tini\pp_{j}^{-1}}\\
&\qqu{in}\Tini \pp_{j+1}^{-1}\cB_{1/4},\qquad\forall\,j\in[1,k-1].
\end{aligned} \]
Since \eqref{cond:rk} implies $\cB_{\ds/2}(\xs)\subset \Tin^{-1}\pp_{j+1}^{-1}\cB_{1/4}$ for all $j\in[1,k-1]$, summing the above estimates  and also using \eqref{Pog0:3ii}  to bound $\big|\rD^3(\cF_{\Tini}w_1)\big|$ in $\Tini \pp_{1}^{-1}\cB_{1/4}$, we find
 \[\begin{aligned}
 &\big|\rD^3 \uu\big|_{\cC^0(\cB_{\ds/2}(\xs))} = \big|\rD^3 \big(\cF_{\Tini}w_k\big)\big|_{\cC^0(\cB_{\ds/2}(\xs))} \le \Qin +\Qin\,  \cE_{\ds}^{2}\,\sum_{j=1}^{k-1}\delta_{j}\,\big/\,\mat{\Tini\pp_{j}^{-1}}\,.
 \end{aligned}
 \]
 Then, by a similar argument to the one between \eqref{dk:bound} and \eqref{def:beta}, we show
 \be\label{l2:t2}
\big|\rD^3 \uu\big|_{\cC^0(\cB_{\ds/2}(\xs))} 
\le \Qin + \Qin  \, \cE_{\ds}^{2}\,\int_{\bar C\ds}^{\qin}{\mc_ f (q)}\,\frac{\rd q}{q^2}\,.
\ee

The first term on the right-hand side of \eqref{bridge:est} and also $|\rD^2 v(\xs)|$  are estimated as follows. By \eqref{Dini:M:inf}, we verify assumption \eqref{Dini:alt:as} of Lemma \ref{lem:Dini}. Thus, \eqref{Dini:cons:1} and \eqref{Dini:cons:2} follows, the latter of which together with  \eqref{vv}, \eqref{Dini:M:inf} shows $|\rD^2 v(\xs)|\le \Qin\,\cE_0$ in terms of $v$, proving the $\cC^2$ bound \eqref{C2:est}.
By factoring out $M_{k+1}$ on the right-hand side of \eqref{Dini:cons:1} and estimating the remaining factor using simply inequality $e^z-1\le e^zz$ for $z\ge0$, we have 
\[
\big|{\rD^2 V(0)}-\rD^2w_k(0)\big|<C\,M_\infty\,\sum_{j=k}^{\infty}\delta_{j}\,.
\] The sum can be estimated  using a similar argument to the one between \eqref{dk:bound} and \eqref{Mk:bd}, except that the limits of the resulting integral are $0$ and $C\mat{\Tini\pp_{k}^{-1}}$. This then prompts us to bound $\mat{\Tini\pp_{k}^{-1}}$ from {\it above} using \eqref{Rr:ratio} and \eqref{Mk:bd}. Also recall  \eqref{uvw:rel} and estimates \eqref{est:t:tau}  on $\tau$. Therefore,
 \[
 \big|{\rD^2 v(\xs)}-b\,\rD^2\uu(\xs)\big|<\Qin\,  \cE_{0}\,\int_0^{\Qin\, \cE_{\ds}\,\ds} {\mc_ f (q)}\,\frac{\rd q}{q}.
 \]

Combining this with \eqref{bridge:est}, \eqref{o:est:2}, \eqref{l2:t2} and \eqref{Mk:bd}, we prove  \eqref{mainestimate}.  

\subsection{First derivative estimate}
Split the difference in the values of $\rD v$ at $\xs$, $\xs'$ and reach the corresponding version of \eqref{bridge:est} where each $\rD^2$ is replaced by $\rD$ and the one $\rD^3$ is replaced by $\rD^2$. In the resulting sum of three terms, the first term is bounded in 
\eqref{back2Thm:1} with $\delta_k$ further bounded in \eqref{dk:d0:bound}, the second terms is bounded using \eqref{back2Thm:2}.
For the third term, since \eqref{back2Thm:0:ww} implies that $ \big|b\,\rD \uu-b'\,\rD\uup\big|$ evaluated at certain point in $\cB_{ \ds /4}(\xm)$ must be bounded by $\Qin\,\cE_{\ds}\,\delta_k\,\ds$, we use \eqref{dk:final} and \eqref{o:est:2} to obtain  
\[
 \big|b\,\rD \uu(\xm)-b'\,\rD\uup(\xm)\big|\le  \Qin\,\cE_{\ds}^{\beta_1}\mc_ f (\Qin\,\cE_{\ds}\, \ds )\,\ds\,.
\] 
Thus, we prove \eqref{mainestimate:1}.

\subsection{Replacing $\mc_ f $ with 
$\mc_{ f |_{\Om_{\rho-\qin}}}$} Observe that almost every $\mc_ f $ in \S \ref{sec:mainproof}  ultimately originates from estimate \eqref{dk:bound} of $\delta_k$. In there, the set $\cB_{\qin}(\xs)\subset \Om_{\rho-\qin}$ and hence $\mc_{ f |_{\cB_{\qin}(\xs)}}(q)\le \mc_{ f |_{\Om_{\rho-\qin}}}(q)$. The only exception is that $\mc_ f $ appears below \eqref{back2Thm:0:ww} in the form of $| f (\xs)- f (\xs')|\le \mc_ f (\ds)$. But $\xs,\xs'\in\Omrho$. Thus, we validate the last statement of Theorem \ref{mainresult}.

\subsection{On the Exponent $\beta$}\label{ss:beta}

This discussion is inspired by \cite{LC5} and \cite{XJW2}.

Start with $\cScmd{\Vo,h_0}$ in \S \ref{sec:iterativeproof}. By \eqref{n:d:2} and the remarks following that, there exist sufficiently small constants $\Cmc',h_{00}$ so that if $|f-f(0)|\le \Cmc'$ in  $\cScmd{\Vo,h_0}$, then $T_0\,\cS(\Vo,h_{00})=\cS(\cF_{T_0}\Vo,1)$ will be sufficiently close to $\cB_{\sqrt{2}}$ so that in the next step of iteration, the upper bound on $|\rD^3w_1|$ can be set as a sufficiently smaller value $C_3'$. This value is so small that we can rely on \eqref{n:d:2}, but with $C_3$ replaced by $C_3'$, to ensure any choice of $\hc'\in[h_{00},\frac12]$ in further iterations can still make $T_1T_0\,\cS(\Vo,h_{00}\hc')=\cS(\cF_{T_1T_0}v,1)$ sufficiently close to $\cB_{\sqrt2}$, and thus keep the iteration going with sequence $\big\{\cS\big(\Vo,h_{00}(\hc')^k\big)\big\}$ while still using the smaller value $C_3'$ as the upper bound for $|\rD^3w_k|$. 

Recall definition \eqref{def:beta} of $\beta$ with $\hc$ replaced by $\hc'$ now. Constants $1.12$ and $\frac65$ both have their origin in \eqref{Tk10:mat}, which shows that they can be made very close to 1 as $\delta_k$ tends to 0 or to a very small $(C_5,\hc')$-dependent value. Then, the $\hc'$-dependent part of \eqref{def:beta} reaches minimum at about $\hc'\approx e^{-1}$. Thus, the $\beta$ value can be made sufficiently close to $2eC_5$. Finally, constant $C_5$ originates from the proof of Lemma \ref{Stepk}; its value can be estimated following \cite[\S\,3.2]{XJW2}.

\appendix

\section{\bf Proof Corollary \ref{cor:1:3}}\label{app:cor:1:3}
\begin{proof}   For existence and uniqueness of the convex Alexandrov solution $v\in \cC^0(\overline\Om)$, we refer to \cite[Thm. 2.13]{AFB2} which does {\it not} requiring {strict} or uniform convexity of $\Om$. 
  
Define 
\[
 f_{\textnormal{ext}}(x):=\begin{cases}
f(x)&\qquad\text{if }\,x\in\Om,\\
1&\qquad\text{if }\,x\in\bR^n\backslash\Om, \end{cases}
\] and define,   for integers $i\gg1$,
  \[  
  f_i(x):=  \int_{\cB_{1/i}} f_{\textnormal{ext}}(x-y)\,\phi (iy)\,i\,\rd y,\qquad\forall\,x \in\Om, 
  \]
  where smooth function $\phi$ is non-negative, supported on the unit ball and integrates to unit. Then, $f_i$ is smooth in $\overline\Om$ and satisfies
 \be\label{cond:fi:1}
1\le f_i(x)\le \fm:=|f|_{L^\infty(\Om)}\qqu{in}\Om.
\ee
Further, since  $f_{\textnormal{ext}}(x-y)=f(x-y)$ if $x\in\Om_{\rho/2} $, $|y|<1/i$ and $i>2/\rho$, we have 
 \be\label{cond:fi:2}
 \mc_{f_i|_{\Om_{\rho/2}}}(q)\le\mc_{f}(q)\,,\quad\qquad\forall\,i>2/\rho.
\ee 
Consider a sequence of MAEs,
   \[\begin{aligned}
    \detn \mathrm{D}^2v_i &= f_i \qu{    in    }\Om, \\
      v_i &=0 \qu{    on    }\pa \Om .
 \end{aligned}
  \]
   The existence and uniqueness of Alexandrov solution $v_i\in \cC^0(\overline\Om)$ has been addressed at the start of the proof. The strict convexity of $v_i$ is due to \eqref{mm:strict} and thus $v_i\in\cC^\infty(\Om)$ by \cite[Thm. 3.10]{AFB2} (we quote this result because it does not require assumptions on $\pa\Om$). 

Now fix any $\rho\in\big(0,\frac12\diam(\Om)\big)$. By assumptions $\om_f(0+)<\min\big\{\Cmc,1/\beta_2\big\}$ and \eqref{cond:g},  the set 
\[
\Big\{z\in(0,\tfrac12\rho]\;\Big|\;\mc_f(z)\le\Cmc\text{ \, and \, }\ \beta_2\,\mc_f(q)< \gamma\text{ for all }q\in(0,z)\Big\}
\]
 is non-empty, thus has a supremum and we  choose it as $\qin$.  Then, by \eqref{cond:fi:2} and $\qin\le\frac12\rho$,  we have $\mc_{f_i|_{\Om_{\rho-\qin}}}\le  \mc_{f_i|_{\Om_{\rho/2}}}(q)\le \mc_f(q)$ and hence having \eqref{choose:first}, upon replacing $f$ with $f_i$, validated for all $i>2/\rho$. By Theorem \ref{mainresult} and noting the last statement therein,  we prove \eqref{mainestimate:1} in terms of the same $\qin,\Qin$ and $\om_{f}(q)$ for every $x,x'\in\Om_\rho$ and every  $(v_i,f_i)$.  We can also relax the explicit term $\mc_{f_i}$ in \eqref{mainestimate:1}  to a constant $(\fm-1)$ due to \eqref{cond:fi:1}, followed by relaxing both exponents of $\cE_{\ds}$ to $\max\{1,\beta_1\}=\beta_2/\beta$. In short,  with $\ds=|x-x'|>0$, we show
\be\label{mainestimate:11}
\dfrac{\big|\rD v_i(x)-\rD v_i (x')\big|}\ds
\le
\begin{cases}    
\Qin \,\cE_{\ds}^{\beta_2/\beta}&\text{ \; if  \; }  \Qin\,\ds < 1,\\
\Qin &\text{ \; if \; } \Qin\,\ds\ge1,
\end{cases}\qquad\forall\,x,x'\in\Om_\rho,
\ee where  the same $(n,\fm,\rho)$-dependent values of constants apply to all $i>2/\rho$. 

Note that, by the above choice of $\qin$, we have
\be\label{ds:g:E}
\ds^{\,\gamma}   \cE_{\ds}^{\beta_2/\beta}\le (\qin/\bar C)^{\gamma}\exp\Big( \int_{ \bar C\ds}^{\qin}\left(\beta_2\,\mc_f(q)-\gamma\right)\,\frac{\rd q}{q}\Big)\le  (\qin/\bar C)^{\gamma}\qqu{if}\bar C\ds\le\qin.
\ee  
Therefore, in view of \eqref{mainestimate:11}, we prove that, for all $i>2/\rho$,
\be\label{mainestimate:1g}
\dfrac{\big|\rD v_i(x)-\rD v_i (x')\big|}{\ds^{1-\gamma}}\le \text{an $(n,\fm,\rho,\om_f,\gamma)$-dependent constant}.
\ee

On the other hand, combining the global H\"older estimate of Corollary \ref{cor:sep} and the Arz\'ela-Ascoli lemma, we  can find a subsequence, also named $\{v_i\}$ for simplicity, so that
\be\label{v:inf}
\lim_{i\to\infty}\textstyle\max_{\,\overline\Om}|v_i-v_\infty|=0\qquad\text{for a convex function $v_\infty$}.
\ee
Moreover, by  a standard stability result (e.g., \cite[Lem. 1.2.3]{CG1}, \cite[Prop. 2.6]{AFB2}), we have the weak-* convergence of the sequence of Monge-Amp\`ere measures $\big\{\!\det \mathrm D^2v_i\rd x\big\}$  to a limit Monge-Amp\`ere measure $\cM v_\infty$ in the domain $\Om$. Now, by standard functional analysis, we have  strong convergence $\|f_i-f\|_{L^p(\Om)}\to0$ for any $p\in[1,\infty)$ which in turn implies the weak-* convergence of $\big\{\!\det \mathrm D^2v_i\rd x=f_i^n\rd x\big\}$ to $f^n\rd x$ as $i\to\infty$ in the domain $\Om$. By uniqueness of weak-* limit, we prove  
\(
\cM v_\infty=f^n\rd x.
\)
Also, apparently $v_i\big|_{\pa\Om}=0$ implies $v_\infty\big|_{\pa\Om}=0$. Therefore, we show that $v=v_\infty$ is the   convex Alexandrov solution  to \eqref{cM:v}.

Finally, since \eqref{mainestimate:1g} apparently extends to $\overline\Omrho$, we use Arz\'ela-Ascoli lemma to find a subsequence of $\big\{\rD v_i \big\}$ converge uniformly on $\overline\Omrho$ to certain limit as $i\to\infty$. By \eqref{v:inf}, this limit has to be $\rD v_\infty$. Therefore, we prove $\rD v_\infty$ satisfies the same  the H\"older estimate \eqref{mainestimate:1g} and hence $v=v_\infty$ is in $\cC^{1,1-\gamma}(\overline\Omrho)$. By the arbitrariness of $\rho$, we complete the proof of $v\in \cC^{1,1-\gamma}_{\textnormal{loc}}({\Om})$ for all $\gamma$'s satisfying \eqref{cond:g}. 
If further $\int_{0}^{1}\left(\mc_f(q)-\om_f(0+)\right)\,\frac{\rd q}{q}<\infty$ holds, then we choose $\qin$ so that $\mc_f(\qin)\le\Cmc$. The middle expression of \eqref{ds:g:E} with
 $\gamma=\beta_2\,\om_f(0+)$ is bounded due to that extra integrability condition, and the same argument works to show  $v\in \cC^{1,1-\gamma}_{\textnormal{loc}}({\Om})$.  
\end{proof}

{\bf Example 1}. Sketch proof. Let $\sigma_\ep:=\sigma_0+\ep$  for a fixed $\ep>0$. By definition of $\sigma_0$, the set 
\[
\Big\{z\in(0,\tfrac12\rho]\;\Big|\; \mc_f(z)\le\Cmc\text{ \, and \, }\beta\,\mc_f(q)\,|\ln q|<\sigma_\ep,\;\forall\,q\in(0,z)\Big\}
\]
 is non-empty, thus has a supremum and choose it as $\qin$. Then, 
\[
\cE_{\ds}\le O(|\ln\ds|^{\sigma_\ep})\qqua{and thus}\cE_{\ds}^{\beta_1}\mc_f(\Qin \, \cE_{\ds}\,\ds ) \le O\big(|\ln\ds|^{\sigma_\ep\beta_1-1}\big).
\]

{\bf Example 2}: Sketch proof. The trick to control the second integral in \eqref{mainestimate} is to find  $\qin<1$ satisfying $a\le\frac12|\ln \qin|$. Then, upon integrating by parts,
\[
\int_{\bar C\ds}^{\qin}{|\ln q|^{-a}}\,\frac{\rd q}{q^2}=-\frac{|\ln q|^{-a}}{q}\Big|_{\bar C\ds}^{\qin}+\int_{\bar C\ds}^{\qin}{a|\ln q|^{-a-1}}\,\frac{\rd q}{q^2}\,,
\]
we find the integral on the right-hand side does not exceed $\frac12$ times the left-hand side.


\begin{thebibliography}{FJM16}

\bibitem[Bur78]{Burch}
Charles~C. Burch.
\newblock The Dini Condition and Regularity of Weak Solutions of Elliptic Equations.
\newblock {\em J. Diff. Eq.}, 30:308-323, 1978.

\bibitem[Caf90a]{LC4}
Luis~A. Caffarelli.
\newblock A localization property of viscosity solutions to the
  {M}onge-{A}mp\`ere equation and their strict convexity.
\newblock {\em Ann. of Math.}, 131:129-134, 1990.

\bibitem[Caf90b]{LC5}
Luis~A. Caffarelli.
\newblock Interior $W^{2,p}$ estimate for solutions of the {M}onge-{A}mp\`ere
  equation.
\newblock {\em Ann. of Math.}, 131:135-150, 1990.

\bibitem[Caf91]{C:C1alpha}
Luis~A. Caffarelli.
\newblock Some regularity properties of solutions of {M}onge-{A}mp\`ere
  equation.
\newblock {\em Comm. Pure Appl. Math.}, 44:965-969, 1991.

\bibitem[CO19]{BCTO}
Bin Cheng and Thomas O'Neill.
\newblock An interior a priori estimate for solutions to {M}onge-{A}mp\`ere
  equations with right-hand side close to one. {\em preprint}, 2019.
\newblock \href{https://arxiv.org/abs/1911.03120}{https://arxiv.org/abs/1911.03120}

\bibitem[Fig17]{AFB2}
Alessio Figalli.
\newblock {\em The {M}onge-{A}mp\`ere Equation and its Applications}.
\newblock European Mathematical Society, 2017.

\bibitem[FJM16]{AFCM1}
Alessio Figalli, Yash Jhaveri and Connor Mooney.
\newblock Nonlinear bounds in {H}{\"o}lder spaces for the {M}onge-{A}mp{\`e}re
  equation.
\newblock {\em J. Func. Anal.}, 270:3808-3827, 2016.

\bibitem[GT01]{GTB}
David Gilbarg and Neil~S. Trudinger.
\newblock {\em Elliptic Partial Differential Equations of Second Order}. (reprint of the 1998 ed.)
\newblock Springer-Verlag, 2001.

\bibitem[Gut16]{CG1}
Cristian~E. Guti\'errez.
\newblock {\em The {M}onge-{A}mp\`ere Equation}. (2nd ed.)
\newblock Birkh\"auser, Boston, 2016.


\bibitem[Joh48]{JohnF}
Fritz John.
\newblock Extremum problems with inequalities as subsidiary conditions.
 In {\em  Studies and Essays Presented to R. Courant on his 60th Birthday, January 8, 1948.} pp. 187-204. Interscience, New York, 1948.


\bibitem[JW07]{XJW2}
Huai-Yu Jian and Xu-Jia Wang.
\newblock Continuity estimates for the {M}onge-{A}mp\`ere equation.
\newblock {\em SIAM J. Math. Anal.}, 39:608-626, 2007.

\bibitem[Kov97]{Kovats}
Jay Kovats.
\newblock Fully nonlinear elliptic equations and the {D}ini condition.
\newblock {\em Comm. Part. Diff. Eq.},
  22:1911-1927, 1997.


\bibitem[Pog71]{P:est}
Aleksei~V. Pogorelov.
\newblock The regularity of the generalized solutions of the equation {${\rm
  det}(\partial ^{2}u/\partial x^{i}\partial x^{j})=$} {$\varphi
  (x^{1},\,x^{2},\dots, x^{n})>0$}.
\newblock {\em Dokl. Akad. Nauk SSSR}, 200:534-537, 1971.



\end{thebibliography}

\end{document}